\documentclass{amsart}

\usepackage[pdfauthor={Dimitrios Chatzakos},
pdftitle={},
            pdfkeywords={Hyperbolic lattice points},
             pdfcreator={Pdflatex}]
{hyperref}
\hypersetup{colorlinks=true}
\usepackage{xcolor}
\usepackage{enumerate}

\usepackage{amssymb,amsmath,latexsym,amsthm}
\usepackage[english]{babel}

\newtheorem{theorem}{Theorem}[section]
\newtheorem{lemma}[theorem]{Lemma}
\newtheorem{proposition}[theorem]{Proposition}

\theoremstyle{definition}
\newtheorem{definition}[theorem]{Definition}

\numberwithin{equation}{section}
\DeclareMathOperator*{\vol}{vol}

\begin{document}

\newtheorem{remark}[theorem]{Remark}



\def \g {{\gamma}}
\def \G {{\Gamma}}
\def \l {{\lambda}}
\def \a {{\alpha}}
\def \b {{\beta}}
\def \f {{\phi}}
\def \r {{\rho}}
\def \R {{\mathbb R}}
\def \H {{\mathbb H}}
\def \N {{\mathbb N}}
\def \C {{\mathbb C}}
\def \Z {{\mathbb Z}}
\def \F {{\Phi}}
\def \Q {{\mathbb Q}}
\def \e {{\epsilon }}
\def \ev {{\vec\epsilon}}
\def \ov {{\vec{0}}}
\def \GinfmodG {{\Gamma_{\!\!\infty}\!\!\setminus\!\Gamma}}
\def \GmodH {{\Gamma\backslash\H}}
\def \sl  {{\hbox{SL}_2( {\mathbb R})} }
\def \psl  {{\hbox{PSL}_2( {\mathbb R})} }
\def \slz  {{\hbox{SL}_2( {\mathbb Z})} }
\def \pslz  {{\hbox{PSL}_2( {\mathbb Z})} }
\def \L  {{\hbox{L}^2}}

\newcommand{\norm}[1]{\left\lVert #1 \right\rVert}
\newcommand{\abs}[1]{\left\lvert #1 \right\rvert}
\newcommand{\modsym}[2]{\left \langle #1,#2 \right\rangle}
\newcommand{\inprod}[2]{\left \langle #1,#2 \right\rangle}
\newcommand{\Nz}[1]{\left\lVert #1 \right\rVert_z}
\newcommand{\tr}[1]{\operatorname{tr}\left( #1 \right)}

\title[Hyperbolic lattice-point counting]{Mean value results and $\Omega$-results for the hyperbolic lattice point problem in conjugacy classes}
\author{Dimitrios Chatzakos}
\address{Universit\'e de Lille 1 Sciences et Technologies
        and
        Centre Européen pour les Mathématiques, la Physique et leurs interactions (CEMPI),
        Cit\'e Scientifique, 59655 Villeneuve d’ Ascq C\'edex, France}
\email{Dimitrios.Chatzakos@math.univ-lille1.fr}
\thanks{The author was supported by a DTA from EPSRC during his PhD studies at UCL}
\date{\today}

\begin{abstract}
For $\G$ a cofinite Fuchsian group, we study the lattice point problem in conjugacy classes on the Riemann surface $\GmodH$.
Let $\mathcal{H}$ be a hyperbolic  conjugacy class in $\Gamma$ and $\ell$ the $\mathcal{H}$-invariant closed geodesic on the surface.
The main asymptotic for the counting function of the orbit $\mathcal{H} \cdot z$ inside a circle of radius $t$ centered at $z$ grows like $c_{\mathcal{H}} e^{t/2}$. This problem is also related with counting distances of the orbit of $z$ from the geodesic $\ell$. For $X \sim e^{t/2}$ we study mean value and $\Omega$-results for the error term $e(\mathcal{H}, X ;z)$ of the counting function. We prove that a normalized version of the error $e(\mathcal{H}, X ;z)$ has finite mean value in the parameter $t$. Further, we prove that if $\Gamma$ is cocompact then
\begin{eqnarray*}
\int_{\ell} e(\mathcal{H}, X;z) d s(z) = \Omega \left( X^{1/2} \log \log \log X \right).
\end{eqnarray*}
For $\Gamma =  {\hbox{PSL}_2( {\mathbb Z})}$ we prove the same $\Omega$-result, using a subconvexity bound for the Epstein zeta function associated to an indefinite quadratic form in two variables. We also study pointwise $\Omega_{\pm}$-results for the error term. Our results extend the work of Phillips and Rudnick for the classical lattice problem to the conjugacy class problem. 

\end{abstract}
\maketitle
\section{Introduction}\label{Introduction}

\subsection{Mean value and $\Omega$-results for the classical hyperbolic lattice point problem} \label{subsectiononeone} Let $\mathbb{H}$ be the hyperbolic plane, $z$, $w$ two fixed points in $\mathbb{H}$ and $\rho(z,w)$ their hyperbolic distance. For $\Gamma$ a cocompact or cofinite Fuchsian group, the classical hyperbolic lattice point problem asks to estimate the quantity
\begin{displaymath}
N(X; z,w)= \# \left\{ \gamma \in \Gamma : \rho( z, \gamma w) \leq \cosh^{-1} \left(\frac{X}{2}\right) \right\},
\end{displaymath}
as $X \to \infty$. This problem has been extensively studied by many authors \cite{cham2, cherubinirisager, delsarte, good, gunther, hillparn, patterson, phirud, selberg}. One of the main methods to understand this problem is using the spectral theory of automorphic forms. For this reason, let $\Delta$ be the Laplacian of the hyperbolic surface $\GmodH$ and let $ \{u_j \}_{j=0}^{\infty}$ be the $\L$-normalized eigenfunctions (Maass forms) of $-\Delta$ with eigenvalues $ \{\lambda_j \}_{j=0}^{\infty}$. We also write $\lambda_j= s_j(1-s_j) =1/4 +t_j^2$. Selberg \cite{selberg}, G\"unther \cite{gunther}, Good \cite{good} et. al. proved that
\begin{equation} \label{mainformulaclassical}
N(X;z,w)  =  \sum_{1/2 < s_j \leq 1} \sqrt{\pi} \frac{\Gamma(s_j - 1/2)}{\Gamma(s_j + 1)} u_j(z) \overline{u_j(w)} X^{s_j} + E(X;z,w),
\end{equation}
where the error term $E(X;z,w)$ satisfies the bound
 \begin{eqnarray*}
 E(X;z,w) = O(X^{2/3}).
\end{eqnarray*}
Conjecturally, the optimal upper bound for the error term $E(X;z,w)$ is
\begin{equation} \label{conjecture1}
E(X; z,w) = O_{\epsilon}(X^{1/2 + \epsilon})
\end{equation}
 for every $\epsilon > 0$ (see \cite{patterson}, \cite{phirud}). This error term has a spectral expansion over all $\lambda_j \geq 1/4$.  The contribution of $\lambda_j=1/4$ is well understood. We subtract it from $E(X;z,w)$ and we define the refined error term $e(X;z,w)$ to be the difference
\begin{eqnarray*}
e(X;z,w) = E(X;z,w) - h(0) \sum_{t_j=0} u_j(z) \overline{u_j(w)}  =  E(X;z,w) + O ( X^{1/2} \log X ),
\end{eqnarray*}
where $h(t)$ is the Selberg/Harish-Chandra transform of the characteristic function $\chi_{[0, (X-2)/4]}$ (see \cite[p. 2]{chatz} for the details).
Thus, bound (\ref{conjecture1}) is equivalent with the bound
\begin{equation} \label{conjecture2}
e(X; z,w) = O_{\epsilon}(X^{1/2 + \epsilon})
\end{equation}
for every $\epsilon > 0$.
For $z=w$, Phillips and Rudnick proved mean value results and $\Omega$-results (i.e. lower bounds for the $\limsup |e(X;z,z)|$) that support conjecture (\ref{conjecture2}). For $\G$ cofinite but not cocompact, let $E_{\frak{a}}(z,s)$ be the nonholomorphic Eisenstein series corresponding to the cusp $\frak{a}$. Phillips and Rudnick \cite{phirud} proved the following theorems.
\begin{theorem} [Phillips-Rudnick \cite{phirud}]\label{philrudn1} (a) Let $\G$ be a cocompact group. Then:
 \begin{equation} \label{mnrt2philrud}
\lim_{T \to \infty} \frac{1}{T} \int_{0}^{T}  \frac{e(2 \cosh r;z,z)}{e^{r/2}} dr = 0.
\end{equation}
(b) Let $\G$ be a cofinite but not cocompact group. Then:
 \begin{equation} \label{mnrt3}
\lim_{T \to \infty} \frac{1}{T} \int_{0}^{T}  \frac{e(2 \cosh r;z,z)}{e^{r/2}} dr = \sum_{\frak{a}}\left| E_{\frak{a}} (z, 1/2) \right|^2.
\end{equation}
\end{theorem}
\begin{theorem} [Phillips-Rudnick \cite{phirud}]\label{philrudn2} (a) If $\G$ is cocompact or a subgroup of finite index in $\pslz$, then for all $\delta >0$, 
 \begin{displaymath}
e(X;z,z) = \Omega \left(X^{1/2} (\log \log X)^{1/4 - \delta} \right).
\end{displaymath}
(b) If $\G$ is cofinite but not cocompact, and either has some eigenvalues $\lambda_j >1/4$ or some cusp $\mathfrak{a}$ with $E_{\frak{a}} (z, 1/2) \neq 0$, then, 
 \begin{displaymath}
e(X;z,z) = \Omega \left(X^{1/2} \right).
\end{displaymath}
(c) If any other cofinite case, for all $\delta >0$, 
 \begin{displaymath}
e(X;z,z) = \Omega \left(X^{1/2-\delta} \right).
\end{displaymath}
\end{theorem}
In the proof of Theorem \ref{philrudn2}, the assumption $z=w$ is essential. In \cite{chatz}, we studied $\Omega$-results for the average
\begin{equation} \label{mtzdefinition}
M(X;z,w) =   \frac{1}{X} \int_{2}^{X}  \frac{e(x;z,w)}{x^{1/2}} d x 
\end{equation}
for two different points $z,w$. We proved that, if $\lambda_1>2.7823...$ and $z,w$ are sufficiently close to each other, the limit of $M(X;z,w)$ as $X \to \infty$ does not exist. In many cases, these results imply pointwise $\Omega$-results for $e(X;z,w)$ with $z \neq w$ as immediate corollaries.

 There are specific groups $\Gamma$ for which we can provide refined $\Omega$-results. In \cite{chatz}, we proved that if $\Gamma$ is a cofinite group with sufficiently many cusp forms at the point $z$ in the sense that the series
\begin{displaymath}
\sum_{|t_j| < T} |u_j(z)|^2 \gg T^{2}
\end{displaymath}
and satisfies $E_{\frak{a}}(z,1/2) \neq 0$ for some cusp $\frak{a}$ then
\begin{eqnarray*}
e(X;z,w) = \Omega_{\pm} (X^{1/2})
\end{eqnarray*}
for $z$ fixed and $w$ sufficiently close to $z$ (see Corollary 1.9 in \cite{chatz}).

\subsection{The conjugacy class problem} \label{subsectiononetwo}In this paper we are interested in studying mean value results and $\Omega$-results for the hyperbolic lattice point problem in conjugacy classes. In this problem we restrict the action of $\G$ in a hyperbolic conjugacy class $\mathcal{H} \subset \G$; that means $\mathcal{H}$ is the conjugacy class of a hyperbolic element of $\G$. Let $z \in \mathbb{H}$ be a fixed point. The problem asks to estimate the asymptotic behavior of the quantity
\begin{eqnarray*}
N_z(t)  = \# \{ \gamma \in \mathcal{H} : \rho( z, \gamma z) \leq t \},
\end{eqnarray*}
as $t \to \infty$. This problem was first studied by Huber in \cite{huber1, huber2}. The main reason we are interested in this problem is because it is related with counting distances of points in the orbit of the fixed point $z$ from a closed geodesic. This geometric interpretation was first explained by Huber in \cite{huber1} and later in \cite{huber2}. Assume $\mathcal{H}$ is the conjugacy class of the hyperbolic element $g^{\nu}$ with $g$ primitive and $\nu \in \N$. Let also $\ell$ be the invariant closed geodesic of $g$. Then $N_z(t)$ counts the number of $\gamma \in  \langle g \rangle \backslash \Gamma  $ such that $\rho(\gamma z, \ell) \leq t$. Equivalently, assume that $\ell$ lie on $\{yi, y>0\}$ (after conjugation). Let $\mu = \mu(\ell)$ be the length of $\ell$ and let $X$ be given by the change of variable
\begin{equation} \label{changeofvariable}
X = \frac{\sinh(t/2)}{\sinh(\mu/2)} \sim c_{\mathcal{H}} \cdot e^{t/2}.
\end{equation}
Huber's interpretation shows that $N_z(t)$ actually counts $\gamma \in \langle g \rangle \backslash \Gamma$ such that $\cos v \geq X^{-1}$, where $v$ is the angle defined by the ray from $0$ to $\gamma z$ and the geodesic $\{yi, y>0\}$.

Under parametrization (\ref{changeofvariable}) denote $N_z(t)$ by $N(\mathcal{H}, X; z )$. Thus we have
\begin{displaymath}
N(\mathcal{H}, X; z) = \# \left\{\gamma \in \mathcal{H} : \frac{\sinh(\rho(z,\gamma z)  /2)}{\sinh(\mu/2)}\leq X \right\}.\end{displaymath}
The conjugacy class problem holds also a main formula similar to formula (\ref{mainformulaclassical}), which can be proved using the spectral theorem for $\L(\GmodH)$. This formula was first derived by Good in \cite{good}; it can also be written in the following explicit form, see \cite{chatzpetr}.
\begin{theorem}[Good \cite{good}, Chatzakos-Petridis \cite{chatzpetr}] \label{mainformulaconjugacy} Let $\Gamma$ be a cofinite Fuchsian group and $\mathcal{H}$ a hyperbolic conjugacy class of $\G$. Then:
\begin{eqnarray*}
N(\mathcal{H}, X;z) &=& \sum_{1/2 < s_j \leq 1}  A(s_j) \hat{u}_j u_j(z) X^{s_j} +  E(\mathcal{H}, X;z),
\end{eqnarray*}
where $A(s)$ is the product:
\begin{equation}\label{a-function} A(s) = 2^{s} \cos \left(\frac{\pi (s-1) }{2} \right) \frac{\Gamma \left(\frac{s+1}{2} \right) \Gamma \left(1 - \frac{s}{2} \right) \Gamma \left(s-\frac{1}{2} \right)}{ {\pi}\Gamma(s+1)},
\end{equation}
\begin{equation} \label{periodintegral}
\hat{u}_j = \int_{\sigma} \overline{u}_j(z) ds(z)
\end{equation} 
is the period integral of $\overline{u}_j$ along a segment $\sigma$ of the invariant closed geodesic of $\mathcal{H}$ with length $\int_{\sigma} ds(z) = \mu / \nu$ and
 \begin{displaymath}
 E(\mathcal{H}, X;z) = O(X^{2/3}).
\end{displaymath}
\end{theorem}
Notice that Theorem \ref{mainformulaconjugacy} implies the main asymptotic of $N(\mathcal{H}, X;z)$ is
$$N(\mathcal{H}, X;z) \sim \frac{2}{\vol (\GmodH)} \frac{\mu}{\nu} X.$$
Once again we are interested in the growth of the error term. The similarities that arise between the two problems suggest that we should expect the bound
\begin{eqnarray}\label{conjecture3}
E(\mathcal{H}, X;z) = O_{\epsilon} (X^{1/2+\epsilon}).
\end{eqnarray}
(see \cite[Conjecture~5.7]{chatzpetr}). As in the classical problem, the error term $E(\mathcal{H}, X;z)$ has a {\lq spectral expansion\rq} over the eigenvalues $\lambda_j \geq 1/4$. We subtract the contribution of the eigenvalue $\lambda_j=1/4$ and we denote the expansion over the eigenvalues $\lambda_j>1/4$ by $e(\mathcal{H}, X;z)$ (eq. (\ref{smalleerror})). In section \ref{spectraltheoryoftheconjugacyclassproblem} we prove that the bound (\ref{conjecture3}) is equivalent with the bound
 \begin{equation} \label{conjecture4}
e(\mathcal{H}, X;z) = O_{\epsilon} (X^{1/2+\epsilon}).
\end{equation}
In order to state our first result, we will need the following definition. 
\begin{definition} \label{definitioneisensteinperiods}
The Eisenstein period associated to the hyperbolic conjugacy class $\mathcal{H}$ is the period integral
\begin{equation} 
\hat{E}_{\mathfrak{a}} (1/2+ it)=\int_{\sigma} E_{\mathfrak{a}} (z, 1/2- it) ds(z),
\end{equation}
across a segment $\sigma$ of the invariant geodesic $\ell$ with length $\int_{\sigma} ds(z) = \mu / \nu$.
\end{definition} 
In section \ref{section3} we prove that the error term $e(\mathcal{H}, X;z)$ has finite mean value in the radial parameter $t$.
\begin{theorem} \label{result1} Let $\Gamma$ be a cofinite Fuchsian group.
\\
(a) If $\G$ is cocompact, then
 \begin{equation} \label{mnrt1}
\lim_{T \to \infty} \frac{1}{T} \int_{0}^{T}   \frac{e \left(\mathcal{H}, e^r ; z\right)}{e^{r/2}}  d r =0.
\end{equation} 
\\(b) If $\G$ is cofinite but not cocompact, then
 \begin{equation} \label{mnrt2}
\lim_{T \to \infty} \frac{1}{T} \int_{0}^{T}  \frac{e \left(\mathcal{H}, e^r ; z\right)}{e^{r/2}} d r = \frac{|\Gamma(3/4)|^2}{\pi^{3/2}} \sum_{\mathfrak{a}} \hat{E}_{\mathfrak{a}} (1/2) E_{\mathfrak{a}} (z, 1/2).
\end{equation}
\end{theorem}

\begin{remark} Using the change of variables (\ref{changeofvariable}) we see that Theorem \ref{result1} is indeed a mean value result in the radial parameter $t \sim 2r + \mu -2 \log 2$.  (where the parameter $t$ counts the distance between the closed geodesic of $\mathcal{H}$ and the orbit of $z$).
\end{remark}  

For the conjugacy class problem, proving pointwise $\Omega$-results is a more subtle problem comparing to the classical one, due to the appearance of the period integrals in the spectral expansion of $e(\mathcal{H}, X;z)$. 
In the proof of Theorem \ref{philrudn2}, Phillips and Rudnick choose $z=w$ so that the series expansion of the error term $e(X;z,w)$ contains the expressions $|u_j(z)|^2$ which are nonnegative. In this setting, the natural choice is to average over the $\mathcal{H}$-invariant geodesic $\ell$. For this reason, we will need the following result of Good and Tsuzuki which describes the exact asymptotic behaviour of the period integrals.
\\
\begin{theorem}[Good \cite{good}, Tsuzuki \cite{tsuzuki}] \label{localweylslawperiods2}
The period integrals $\hat{u}_j$ of Maass forms and $\hat{E}_{\mathfrak{a}} (1/2 + it)$ of Eisenstein series satisfy the asymptotic
\begin{eqnarray*}
 \sum_{|t_j| < T} |\hat{u}_j|^2   + \sum_{\frak{a}} \frac{1}{4 \pi} \int_{-T}^{T} | \hat{E}_{\mathfrak{a}} (1/2 + it)|^2 dt \sim \frac{\mu(\ell)}{\pi} \cdot T,
\end{eqnarray*}
where $\mu(\ell)$ denotes the length of the invariant closed geodesic $\ell$.
\end{theorem} 
We refer to \cite[p.~3-4]{martin} for a detailed history of this result. We also give the following definition which is related to Theorem \ref{localweylslawperiods2}. 
\\
\begin{definition} \label{sufficientmanydefinition} Fix $\mathcal{H}$ be a hyperbolic class of a cofinite but not cocompact group $\Gamma$. We say that the group $\Gamma$ has sufficiently small Eisenstein periods associated to $\mathcal{H}$ if for all cusps $\frak{a}$ we have
\begin{eqnarray*}
 \int_{-T}^{T} | \hat{E}_{\mathfrak{a}} (1/2 + it)|^2 dt \ll \frac{T}{(\log T)^{1+\delta}}
\end{eqnarray*}
for a fixed $\delta >0$.
\end{definition}
For the rest of this paper we write $\int_{\mathcal{H}} ds$ to indicate that we average over a segment of the invariant geodesic $\ell$ of length $\mu/\nu$. When $\mathcal{H}$ is the class of a primitive element we get $\nu=1$, hence $\int_{\mathcal{H}} ds = \int_{\ell} ds$.

We distinguish the two cases of $\Omega$-results: if $g(X)$ is a positive function, we write $e(X;z,w) = \Omega_{+}(g(X))$ if
\begin{eqnarray*}
\lim \sup \frac{e(X;z,w)}{g(X)} > 0,
\end{eqnarray*}
 and $e(X;z,w) = \Omega_{-}(g(X))$ if
\begin{eqnarray*}
\lim \inf \frac{e(X;z,w)}{g(X)} < 0.
\end{eqnarray*}
 In section \ref{section4} we prove the following theorem, which is an average $\Omega$-result on the closed geodesic of $\mathcal{H}$.
\\
\begin{theorem} \label{result2} (a) If $\Gamma$ is either (i) cocompact or (ii) cofinite but not cocompact and  has sufficiently small Eisenstein periods associated to $\mathcal{H}$ according to Definition \ref{sufficientmanydefinition}, then
\begin{eqnarray*}
\int_{\mathcal{H}} e(\mathcal{H}, X ;z) ds(z) = \Omega_{+} (X^{1/2} \log \log \log X).
\end{eqnarray*}
\\(b) If  $\Gamma$ is cofinite but not cocompact and either (i) $\hat{u}_j \neq 0$ for at least one $\lambda_j>1/4$ or (ii) $\hat{E}_{\mathfrak{a}} (1/2) \neq 0$ for a cusp $\mathfrak{a}$ then
\begin{eqnarray*}
\int_{\mathcal{H}} e(\mathcal{H}, X ;z) ds(z) = \Omega_{+} (X^{1/2}).
\end{eqnarray*}
\end{theorem}
\begin{remark}
In subsection \ref{arithmeticexampleresults}  we will see that the modular group $\Gamma = {\hbox{PSL}_2( {\mathbb Z})}$ has sufficiently small Eisenstein periods associated to a fixed conjugacy class $\mathcal{H} \subset \Gamma$. This follows from a subconvexity bound on the critical line for an Epstein zeta function associated to $\mathcal{H}$. 
\end{remark}

The asymptotic behaviour for the sums of period integrals in Theorem \ref{localweylslawperiods2} is $c T$, where in local Weyl's law (Theorem \ref{localweylslaw}) we get an asymptotic $c T^2$. If $\Gamma$ is cocompact or cofinite but it has sufficiently small Eisenstein periods associated to $\mathcal{H}$ then 
\begin{equation}
\sum_{|t_j| < T} |\hat{u}_j|^2 \sim \frac{\mu(\ell)}{\pi}  T,
\end{equation}
and summation by parts implies
\begin{equation}
\sum_{|t_j|< T} \frac{|\hat{u}_j|^2}{t_j} \gg \log T.
\end{equation}
In case (a) of Theorem \ref{result2} the triple logarithm should be compared with the extra factor $(\log \log X)^{1/4-\delta}$ in case (a) of Theorem \ref{philrudn2}. The first is a consequence of the asymptotic behaviour of period integrals in Theorem \ref{localweylslawperiods2}, and the second is a consequence of the local Weyl's law.

To prove pointwise $\Omega$-results for $e(\mathcal{H}, X;z)$ we would like to have a fixed pair $(z, \mathcal{H})$ with $e(\mathcal{H}, X ;z)$ large, i.e. a pair $(z, \mathcal{H})$ with a uniform {\lq fixed sign\rq} property of all $\hat{u}_j u_j(z)$. That would allow us to prove a pointwise $\Omega$-result of the form
\begin{eqnarray*}
\limsup_{X} \frac{|e(\mathcal{H}, X;z)|}{X^{1/2}} = \infty.
\end{eqnarray*}
However, Maass forms have complicated behaviour on the surfaces $\Gamma \backslash \mathbb{H}$; for instance, the nodal domains have very complicated shapes. For this reason we have not been able to determine any such specific pair $(z, \mathcal{H})$ with the desired fixed sign property. To overcome this problem we notice that the period integral is the limit of Riemann sums. Starting with a fixed conjugacy class $\mathcal{H}$, a discrete average allows us to prove the existence of at least one point $z=z_{\mathcal{H}}$ for which the error $e(\mathcal{H}, X;z_{\mathcal{H}})$ cannot be small.
\\\\
 We first prove the following proposition for discrete averages. 

\begin{proposition}\label{result3} Let $\mathcal{H}$ be a fixed hyperbolic class in $\Gamma$. If $\Gamma$ is either (i) cocompact or (ii) if $\Gamma$ is as in part (b) of Theorem \ref{result2}, then there exist an integer $K=K_{\mathcal{H}}$ depending only on $\mathcal{H}$ and $z_1, z_2, ... , z_K$ points on $\ell$ such that:
\begin{displaymath}
\frac{1}{K} \sum_{m=1}^{K} e(\mathcal{H}, X;z_{m}) = \Omega_{+} ( X^{1/2} ).
\end{displaymath}
\end{proposition}

In comparison with our results in \cite{chatz}, in order to prove $\Omega_{-}$-results for the error $e(\mathcal{H}, X;z)$ we are lead to  investigate the behaviour of a modification of the average error term
\begin{displaymath}
 \frac{1}{X} \int_{1}^{X} \frac{e(\mathcal{H}, x ;z)}{x^{1/2}} d x
\end{displaymath}
on the geodesic $\ell$. 
\begin{proposition} \label{result4} Let $\Gamma$ be either (i) cocompact or (ii) cofinite but not cocompact, $\hat{u}_j \neq 0$ for at least one $\lambda_j>1/4$ and $\hat{E}_{\mathfrak{a}} (1/2) = 0$ for all cusps $\mathfrak{a}$. Then there exist an integer $K = K_{\mathcal{H}}$ and $z_1, z_2, ... , z_K$ points in $\ell$ such that, as $X \to \infty$:
\begin{displaymath}
\frac{1}{K} \sum_{m=1}^{K} \frac{1}{X} \int_{1}^{X} \frac{e(\mathcal{H}, x ;z)}{x^{1/2}} d x   = \Omega_{-}(1).
\end{displaymath}
\end{proposition}

We deduce the following theorem on pointwise $\Omega$-results for the error term $e(\mathcal{H}, X; z)$ as an immediate corollary of Theorem \ref{result1} and Propositions \ref{result3}, \ref{result4}.

\begin{theorem} \label{result5} Let $\Gamma$ be a Fuchsian group, $\mathcal{H}$ a hyperbolic conjugacy class of $\Gamma$ and $\ell $ the invariant closed geodesic of $\mathcal{H}$. 
 \\
(a) If $\Gamma$ is as in Proposition \ref{result3}, then there exist at least one point $z_{\mathcal{H}}  \in \ell$ such that: 
\begin{displaymath}
e(\mathcal{H}, X;z_{\mathcal{H}})= \Omega_{+} ( X^{1/2}).
\end{displaymath}
(b) If $\Gamma$ is as in Proposition \ref{result4}, then there exists at least one point $z_{\mathcal{H}}  \in \ell$ such that: 
 \begin{displaymath}
 e(\mathcal{H}, X; z_{\mathcal{H}}) = \Omega_{-} (X^{1/2}).
\end{displaymath}
(c) If $\Gamma$ is not cocompact and the sum $\sum_{\mathfrak{a}} \hat{E}_{\mathfrak{a}} (1/2) E_{\mathfrak{a}} (z, 1/2) $ does not vanish then:
\begin{displaymath}
e(\mathcal{H}, X;z)= \Omega ( X^{1/2}).
\end{displaymath}
\end{theorem}

Finally, at the last section, as an application of Theorem \ref{localweylslawperiods2} we obtain upper bounds for the error terms of both the classical problem and the conjugacy class problem on geodesics.
\begin{remark} 
For the proof of Theorem \ref{result2} and Propositions \ref{result3}, \ref{result4} we will crucially need some {\lq fixed-sign\rq} properties of the $\G$-function stated in Lemma \ref{gammalemma2}. We emphasize that the differences in the signs in the two cases of  Lemma \ref{gammalemma2} cause the different signs of our $\Omega$-results.
\end{remark}



\begin{remark}
It follows from Theorem \ref{result5} that in order to prove a pointwise result $e(\mathcal{H}, X ; z )= \Omega (X^{1/2})$ for one point $z$, we must only assume the nonvanishing of one period $\hat{u}_j$. In this case, the sign of our $\Omega$-result can be determined by the vanishing or not of the Eisenstein period integrals. If $\G$ is cocompact or all Eisenstein periods vanish then there exists at least two points $z, w  \in \ell$ such that: 
\begin{equation} 
\begin{split}
e(\mathcal{H}, X;z)= \Omega_{+} ( X^{1/2}), \\
e(\mathcal{H}, X;w)= \Omega_{-} ( X^{1/2}).
\end{split} 
\end{equation}

These Eisentein periods are of particular arithmetic interest; in fact $\hat{E}_{\mathfrak{a}}(1/2)$ is the constant term of the hyperbolic Fourier expansion of $E_{\mathfrak{a}}(z,s)$ (see \cite[section~3.2]{goldfeld}). In the arithmetic case, these periods are associated to special values of Epstein zeta functions (see subsection \ref{arithmeticexampleresults}). We notice that, in principle, it is easier to check the nonvanishing of one period $\hat{E}_{\mathfrak{a}}(1/2)$ than the nonvanishing of the sum $\sum_{\mathfrak{a}} \hat{E}_{\mathfrak{a}} (1/2) E_{\mathfrak{a}} (z, 1/2)$.
\end{remark}

\begin{remark}
Phillips and Rudnick in \cite{phirud} generalized Theorem \ref{philrudn1} and case $c)$ of Theorem \ref{philrudn2} in the case of the $n$-dimensional hyperbolic space $\H^n$ \cite[p.~106]{phirud}. 

Recently, Paarkonen and Paulin \cite{parkpaul} studied the hyperbolic lattice point problem in conjugacy classes for the $n$-th hyperbolic space and in a more general setting. However, their geometric approach cannot be used to generalise our results in dimensions $n \geq 3$. To do this, we need an explicit expression for the Huber transform $d_n (f,t)$ in the $n$-th dimension. In dimension $n=3$, $d_3(f,t)$ was recently studied explicitly by Laaksonen in \cite{laaksonen}, where he obtained upper bounds for the second moments of the error term, generalising previous work by the author and Petridis \cite{chatzpetr}.

\end{remark}

\subsection{Acknowledgments} I would like to thank my supervisor Y. Petridis for his helpful guidance and encouragement. I would also want to thank V. Blomer for bringing to my knowledge the subconvexity bound for the Epstein zeta function of an indefinite quadratic form. Finally, I would like to thank the two anonymous referees for many corrections and their valuable and helpful comments.

\section{Spectral theory and counting} \label{spectraltheoryoftheconjugacyclassproblem}

\subsection{The Huber transform} \label{thehubertransform}

We briefly state the basic results from the spectral theory of automorphic forms for the conjugacy class problem (see \cite[section 2]{chatzpetr} for the details). Let $C_0^*[1, \infty)$ denote the space of real functions of compact support that are bounded in $[1, \infty)$ and have at most finitely many discontinuities.
\begin{definition} Let $f \in C_0^*[1, \infty)$.  The Huber transform $d(f,t)$ of $f$ at the spectral parameter $t$ is defined as
\begin{equation}\label{coefficients}
d(f, t) = \int_{0}^{\frac{\pi}{2}} f \left(\frac{1}{\cos^2 v} \right) \frac{\xi_{\lambda}(v)}{\cos^2 v} dv,
\end{equation}
with $\lambda=1/4+t^2$, and $\xi_{\lambda}$  is the solution of the differential equation
\begin{equation}\label{huberresult}\xi_{\lambda}''(v) + \frac{\lambda}{\cos^2 v} \xi_\lambda(v) = 0, \quad v \in \Big(-\frac{\pi}{2}, \frac{\pi}{2} \Big), \end{equation}
with $\xi_{\lambda}(0)=1$, $\xi_{\lambda}'(0)=0$.
\end{definition}
The Huber transform plays a role analogous to that of the Selberg/Harish-Chandra transform in the classical counting (see \cite{chatzpetr}, \cite{huber2}). For this reason we work with $d(f,t)$ for an appropriate test function $f = f_X$.

\subsection{The test function and counting} \label{testandspecialfunctions}

Assume first that $\GmodH$ is compact .  For an $f \in C_0^{*} [1, \infty)$ we define the $\G$-automorphic function 
\begin{equation}\label{afgeometric}
A(f) (z)= \sum_{\gamma \in \mathcal{H}} f \left( \frac{\cosh \rho(z,\gamma z) -1}{\cosh \mu -1} \right) .
\end{equation}
The following proposition gives the Fourier expansion of the counting function $A(f) (z)$ (see \cite[p.~984]{chatzpetr}, \cite[p.~17]{huber2}).
\begin{proposition} \label{newpropositionreviewerfourierexpansion} The function $A(f)(z)$ has a Fourier expansion of the form
\begin{equation}\label{afspectral}
A(f) (z)= \sum_{j} 2 d (f, t_j) \hat{u}_j u_j(z),
\end{equation}
where $d(f,t)$ is the Huber transform of $f$.
\end{proposition}
The quantity $N(\mathcal{H}, X ;z)$ can be interpreted as
\begin{equation}\label{asequalnx}
A(f_X)(z) = N(\mathcal{H}, X ;z),
\end{equation}
for $f_X = \chi_{[1,X^2]}$, the characteristic function of the interval $[1,X^2]$. We have the following lemma.
\begin{lemma} \label{lemmacoefficentsconjugacy} Let $s=1/2+it$. Let also $U =\sqrt{X^2-1}$, $R= \log (X+U)$ and $r= \log (x+\sqrt{x^2 - 1})$ (thus $X =2 \cosh R$ and $x=2 \cosh r$) and define the function
\begin{eqnarray} \label{notationsofpaper}
G(t) &=&   \frac{2 \sqrt{2}}{\pi}   \frac{| \G (3/4 + it/2) |^2  }{\Gamma(3/2+it)} \cos (i \pi t/2-  \pi/4). 
\end{eqnarray}
Then, for the Huber transform of $f_X$ we have the following estimates.
\\
(a) If $s \in (1/2,1]$ then 
\begin{eqnarray*}
2 d (f_X, t) = A(s) X^s + O \left((s-1/2)^{-1} X^{1-s}\right),
\end{eqnarray*}
where $A(s)$ is the $\G$-product defined in (\ref{a-function}).
\\
(b)For $t \in \R - \{0\}$ we have
\begin{eqnarray*}
2 d (f_X, t) &=&  \Re \left( G(t) \Gamma(it)  e^{i t R} \right) X^{1/2} +  \Re \left( V(R,t)  e^{itR}
\right),
\end{eqnarray*}
with $V(R,t)  = O \left(  (1+|t|)^{-2} X^{-3/2} \right)$.
\\
(c) For $t=0$ we have
\begin{eqnarray*}
d (f_X, 0) = O( X^{1/2} \log X).
\end{eqnarray*}
\end{lemma}
\begin{remark}
Stirling's formula implies that, as $|t| \to \infty$,
\begin{eqnarray} \label{asymptot}
 |G(t) \Gamma(it)| \asymp (1+|t|)^{-1}.
\end{eqnarray}
We can now give the proof of the Lemma.
\end{remark}
\begin{proof}
(a) Using the integral representation for $d(f_X,t)$ in \cite[p.~5]{chatzpetr} we get
\begin{equation}\label{df_xt_0}
d(f_X, t) =  (2 \sqrt{\pi})^{-1} \Gamma \left(\frac{s+1}{2} \right) \Gamma \left(1 - \frac{s}{2} \right) \int_{0}^{U} \left( P_{s-1}^0 (i v) + P_{s-1}^0 (-i v) \right) dv.
\end{equation}
Using \cite[p.~968, eq.~(8.752.3)]{gradry}, this takes the form
\begin{equation}\label{df_xt}
d(f_X, t) =  (2 \sqrt{\pi})^{-1} \Gamma \left(\frac{s+1}{2} \right) \Gamma \left(1 - \frac{s}{2} \right) X \left( P_{s-1}^{-1} (i U) - P_{s-1}^{-1} (-i U) \right).
\end{equation}
Using formula \cite[p.~971, eq.~(8.776)]{gradry}, the statement follows.
\\
(b) We use \cite[p.~971, eq.~(8.774)]{gradry}, so that equation (\ref{df_xt}) gives
\begin{equation} \label{df_hypergeometric}
2 d (f_X, t) = \Re \left( G(t) \G(it) e^{i t R} F \left(-\frac{1}{2}, \frac{3}{2};1+it; \frac{e^{-R}}{2X} \right) \right) X^{1/2},
\end{equation}
where $F(a,b;c;z)$ denotes the Gauss' hypergeometric function. As $X \to \infty$, the definition of the hypergeometric function \cite[p.~1005, eq.~(9.100)]{gradry} implies
\begin{eqnarray} \label{hypergeometricspecialvalue}
F \left(-\frac{1}{2}, \frac{3}{2};1+it; \frac{e^{-R}}{2X} \right) = 1 + O\left( (1+|t|)^{-1} X^{-2}\right).
\end{eqnarray}
The statement of part (b) now follows.
\\
(c) Plugging $t=0$, i.e. $s=1/2$, in eq. (\ref{df_xt}) and and using formula \cite[p.~961, eq.~(8.713.2)]{gradry}, we calculate
\begin{eqnarray*}
P_{-1/2}^{-1}(i U) - P_{-1/2}^{-1}(-i U) &\ll& X^{5/2} \int_{0}^{\infty} \left(\cosh^2 t + U^2 \right)^{-3/2} dt \\
&\ll& X^{-1/2} \int_{0}^{\infty} \left( \left(\frac{\cosh t}{U}\right)^2 + 1 \right)^{-3/2} dt.
\end{eqnarray*}
Setting $x = \cosh t / U$ we get
\begin{eqnarray*}
\int_{0}^{\infty} \left( \left(\frac{\cosh t}{U}\right)^2 + 1 \right)^{-3/2} dt &=&  \int_{1/U}^{\infty} \left(x^2 + 1 \right)^{-3/2} 
\frac{U}{(U^2 x^2 -1 )^{1/2}} dx \\
&=& \int_{1/U}^{1} \left(x^2 + 1 \right)^{-3/2} \frac{U}{(U^2x^2-1)^{1/2}} dx \\
&&+ \int_{1}^{\infty} \left(x^2 + 1 \right)^{-3/2} 
\frac{U}{(U^2x^2 -1)^{1/2}} dx.
\end{eqnarray*}
For $U \geq 2$ we get
$$ \int_{1}^{\infty} \left(x^2 + 1 \right)^{-3/2} \frac{U}{(U^2x^2-1)^{1/2}} dx \ll  \int_{1}^{\infty} \left(x^2 + 1 \right)^{-3/2} dx \ll 1$$
and, after setting $u=x U$,
\begin{eqnarray*}
 \int_{1/U}^{1} \left(x^2 + 1 \right)^{-3/2} \frac{U}{(U^2x^2-1)^{1/2}} dx &=& \int_{1}^{U} \left(\frac{U^2}{u^2+U^2}\right)^{3/2} \frac{U}{(u^2-1)^{1/2}} \frac{du}{U} \\
&\leq&  \int_{1}^{U} \frac{1}{\sqrt{u^2-1}} du \ll \log U \ll \log X.
\end{eqnarray*}
Combining these estimates we get
$$P_{-1/2}^{-1}(i U) + P_{-1/2}^{-1}(-i U) \ll X^{-1/2} \log X.$$
\end{proof}

If we ignore for a while any issue of convergence, then using $(a)$ of Lemma \ref{lemmacoefficentsconjugacy} and Proposition \ref{newpropositionreviewerfourierexpansion} we obtain that, in the compact case, the error term $E(\mathcal{H}, X;z)$ has a formal {\lq spectral expansion\rq} of the form 
\begin{eqnarray*} 
E(\mathcal{H}, X;z) &=&  \sum_{ t_j \in \mathbb{R}}  2 d(f_X , t_j) \hat{u}_j u_j(z) + O \left( \sum_{1/2<s_j \leq 1} (s_j-1/2)^{-1} X^{1-s_j} \right).
\end{eqnarray*}
The $s_j$'s are discrete, thus we can find a constant $\sigma=\sigma_{\G} \in (0,1/2]$ such that $s_j -1/2 \geq \sigma$ for all $s_j \in (1/2,1]$. This implies that the above $O$-term is $O(X^{1/2-\sigma})$. Using $(c)$ of Lemma \ref{lemmacoefficentsconjugacy} and the finiteness of the eigenspace for the eigenvalue $t_j =0$ we get the bound
\begin{displaymath}
 d(f_X, 0) \sum_{t_j=0} \hat{u}_j  u_j(z) = O( X^{1/2} \log X).
\end{displaymath}
Since the contribution of the eigenvalue $\lambda_j=1/4$ is well understood and does not affect the square root cancellation conjecture for the error term, we subtract this quantity from $E(\mathcal{H}, X;z)$ and we define the modified error term $e(\mathcal{H}, X;z)$ to be the difference
\begin{equation} \label{smalleerror}
e(\mathcal{H}, X;z) = E(\mathcal{H}, X;z ) -   d(f_X, 0) \sum_{t_j=0} \hat{u}_j  u_j(z).
\end{equation}
Thus, if we ignore issues of convergence, for $\G$ cocompact we conclude the principal series of the error $e(\mathcal{H}, X;z)$ takes the form:
\begin{equation} \label{almostspectralexpansionconjugacy}
e(\mathcal{H}, X;z) =  \sum_{ t_j >0}  2 d(f_X , t_j) \hat{u}_j u_j(z) + O(X^{1/2-\sigma}).
\end{equation}

\subsection{Some more auxiliary lemmas}\label{lemmasproofs}

One of the key ingredients in the proofs of our results is the following lemma.

\begin{lemma}\label{gammalemma2} For every $t \in \R -\{0\}$, we have:
\\
a) \begin{displaymath}
\Re\left( G(t) \G(it)   \right) > 0,
\end{displaymath}
b) \begin{displaymath}
\Re\left( \frac{G(t) \G(it)}{(1 +it)} \right) < 0.
\end{displaymath}
\end{lemma}

\begin{proof} (of Lemma \ref{gammalemma2})
$a)$ Obviously, the first inequality is equivalent with 
\begin{eqnarray}\label{lemmagammaequation1}
\Re\left( \frac{\G(it)}{\G(3/2+it)} \cos (i \pi t/2-  \pi/4) \right) > 0.
\end{eqnarray}
Since $\G(\overline{z}) = \overline{\G(z)}$, it suffices to prove the lemma for $t >0$. Notice that 
\begin{eqnarray}\label{theexponentialexplicit}
 \cos (i \pi t/2-  \pi/4) = \frac{\cosh \left(\frac{\pi t}{2}\right)}{\sqrt{2}} + \frac{i \sinh \left(\frac{\pi t}{2}\right)}{\sqrt{2}}.
\end{eqnarray}
Using \cite[p.~909, eq.~(8.384.1)]{gradry} we get
\begin{displaymath}
 \frac{\G(it)}{\G(3/2+it)} = \frac{2}{\sqrt{\pi}} B(it, 3/2),
\end{displaymath}
where $B(x,y)$ is the Beta function. By the definition of Beta function \cite[p.~908, eq.~(8.380.1)]{gradry} and the formula
 \begin{displaymath}
B(x+1,y) = B(x,y) \frac{x}{x+y}
\end{displaymath}
we see that inequality (\ref{lemmagammaequation1}) is equivalent with
\begin{eqnarray*}
\Re\left( \frac{\G(it)}{\G(3/2+it)} \right) \cosh \left( \frac{\pi t}{2} \right) -  \Im \left( \frac{\G(it)}{\G(3/2+it)} \right) \sinh \left( \frac{\pi t}{2} \right) >0, 
\end{eqnarray*}
which is equivalent with $Q(t) >0$, where $Q(t)$ is the function defined by
\begin{eqnarray} \label{finalinequality}
\begin{split}
Q(t) &:=& \left( \int_{0}^{1}\cos (t \log s ) (1-s)^{1/2} ds \right) \left(2t + 3  \tanh \left( \frac{\pi t}{2} \right)   \right) \\
&&+    \left( \int_{0}^{1} \sin (t \log s) (1-s)^{1/2} ds \right) \left(3   -2t  \tanh \left( \frac{\pi t}{2} \right) \right).
\end{split}
\end{eqnarray}
From \cite[Lemma~2.2]{chatz} it follows that if $f : (-\infty , 0) \to \mathbb{R}$ is a continuous and strictly decreasing real valued function such that $f(x) \sin(x)$ is integrable in $(-\infty, 0)$, then 
 \begin{eqnarray} \label{lemmafrompaper2}
\int_{-\infty}^{0} f(x)  \sin(x) dx > 0.
\end{eqnarray}
To prove (\ref{finalinequality}) we integrate by parts, we set $s=e^{x/t}$ and we apply (\ref{lemmafrompaper2}) for 
\begin{displaymath}
f_{t}(x)= \frac{1}{t} (1-e^{x/t})^{1/2} e^{x/t} \left(2t + 3  \tanh \left( \frac{\pi t}{2} \right) \right) -  ((1-e^{x/t})^{1/2} e^{x/t})'  \left(2 -\frac{4t}{3}  \tanh \left( \frac{\pi t}{2} \right) \right), 
\end{displaymath}
which can be easily checked to be decreasing for $t \geq 2/\pi$. For $t\leq 2/\pi$, notice that 
\begin{eqnarray}
\Re\left( G(t) \G(it)   \right) =   \frac{4}{\pi^{3/2}}   \left| \G \left(\frac{3}{4} + \frac{it}{2}\right) \right|^2 \frac{\cosh (\pi t/2)}{2} \frac{Q(t)}{t}.
\end{eqnarray}
Taking $t \to 0$ we get $\lim_{t \to 0} Q(t)/t >0$, hence $\lim_{t \to 0} \Re\left( G(t) \G(it)   \right) >0$ and  the lemma holds for $t$ sufficiently small. Taking derivatives, we write $Q'(t)$ in the form $Q'(t) = \int_{-\infty}^{0} g_t(x)  \sin(tx) dx$. Applying \cite[Lemma~2.2]{chatz} to $Q'(t)$ we conclude that $Q(t)$ is increasing. Hence, part $a)$ follows.
Part (b) can be proved along exactly the same lines, using 
\cite[Lemma~2.2]{chatz}.
\end{proof}

We will finally need the following estimate for the Maass forms and the Eisenstein series which is a local version of Weyl's law for $\L(\GmodH)$.
\begin{theorem}[Local Weyl's law] \label{localweylslaw} 
For every $z$, as $T \to \infty$,
\begin{eqnarray*}
\sum_{|t_j| <T} |u_j(z)|^2 + \sum_{\frak{a}} \frac{1}{4\pi} \int_{-T}^{T} | E_{\frak{a}} \left( z, 1/2+it \right) |^2 dt \sim c T^2,
\end{eqnarray*}
where $c=c(z)$ depends only on the number of elements of $\Gamma$ fixing $z$.
\end{theorem}
See \cite[p.~86, lemma~2.3]{phirud} for a proof of this result. We emphasize that if $z$ remains in a compact set of $\mathbb{H}$ the constant $c(z)$ remains uniformly bounded. 

\section{The mean value result} \label{section3}

\subsection{Proof of Theorem \ref{result1} for $\GmodH$ compact} 

We first prove the error term $e(\mathcal{H},X;z)$ has zero mean value for $\Gamma$ cocompact.
\begin{proof}
In this case $\G$ has only discrete spectrum. The characteristic function $f_X$ is not smooth; thus when we apply the spectral theorem  for $\L (\Gamma \backslash \mathbb{H})$ \cite[p.~69, Theorem~4.7 and p.~103, Theorem~7.3]{iwaniec} directly to $A(f_X)$, we deduce the spectral expansion (\ref{almostspectralexpansionconjugacy}). This principal series is not absolutely convergent. To avoid convergence issues, for $x=2 \cosh r \sim e^r$ we use the identity
\begin{eqnarray*}
d \left( \frac{1}{T} \int_{0}^{T} \frac{f_x dr}{e^{r/2}} ,  t  \right) =  \frac{1}{T} \int_{0}^{T} \frac{d (f_x,t)}{e^{r/2}}  dr, 
\end{eqnarray*}
i.e. the Huber transform commutes with multiplication of $f_x$ by a function that depends only on the radial variable $x$, and it commutes with integration over $r$. Hence, if we define the integrated error
\begin{eqnarray} \label{integratederror}
M_{\mathcal{H}}(T) =  \frac{1}{T} \int_{0}^{T}  \frac{e \left(\mathcal{H}, x ; z\right)}{x^{1/2}} d r,
\end{eqnarray}
this has the spectral expansion
\begin{eqnarray*}
M_{\mathcal{H}}(T) =\sum_{ t_j >0} 2  \hat{u}_j u_j(z)    \frac{1}{T} \int_{0}^{T}   \frac{d (f_x,t)}{x^{1/2}}   d r +O(T^{-\sigma}).
\end{eqnarray*}
Using part $(b)$ of Lemma \ref{lemmacoefficentsconjugacy} we conclude
\begin{eqnarray*}
M_{\mathcal{H}}(T) &=&
\sum_{ t_j >0}     \Re \left( G(t_j) \G(it_j) \frac{1}{T} \int_{0}^{T}  e^{i t_jr}d r  \right)  \hat{u}_j u_j(z)      \\
&&+ O \left(\sum_{ t_j >0}  \frac{|\hat{u}_j| |u_j(z)|}{T} \left| \int_{0}^{T} \frac{V(r,t_j) }{x^{1/2}} e^{it_jr}dr \right| +   \frac{1}{T} \int_{0}^{T} e^{-r \sigma} dr \right).
\end{eqnarray*}
Using Theorems \ref{localweylslawperiods2}, \ref{localweylslaw} and Stirling's formula (estimate (\ref{asymptot})) we bound the main term by $O(T^{-1})$. For the first summand in the $O$-term we use integration by parts. Using that $V(R,t)$ is given by the formula
\begin{eqnarray} \label{definitionv}
V(r,t) = G(t) \Gamma(it) \left(F \left(-\frac{1}{2}, \frac{3}{2};1+it; \frac{e^{-r}}{2x} \right) - 1 \right) x^{1/2}
\end{eqnarray}
and using trivial estimates for the derivative of the hypergeometric function we obtain
\begin{eqnarray*}
 \int_{0}^{T} \frac{V(r,t_j) }{x^{1/2}} e^{it_jr}dr = O(t_j^{-2}).
\end{eqnarray*}
Hence the $O$-terms are also bounded by $T^{-1}$, and the statement follows. 
\end{proof}

\subsection{Proof of Theorem \ref{result1} for $\Gamma$ for cofinite} In this case the hyperbolic Laplacian $-\Delta$ has also continuous spectrum which is spanned by the Eisenstein series $E_{\frak{a}}(z, 1/2+it)$ (see \cite[chapters 3,6 and 7]{iwaniec}). To prove case $(b)$ of Theorem \ref{result1} we have to consider the contribution of the continuous spectrum in $M_{\mathcal{H}}(T)$, which is given in terms of the Eisenstein series $ E_{\mathfrak{a}} (z, 1/2+ it)$ and the period integrals $\hat{E}_{\mathfrak{a}} (1/2+ it)$. More specifically, using \cite[eq.~(4.1)]{chatzpetr} and \cite[Lemma~4.2]{chatzpetr} we get that the contribution of the continuous spectrum is given by
\begin{equation} \label{continuouscontributionconj}
\sum_{\mathfrak{a}} \frac{1}{4\pi} \int_{-\infty}^{\infty}  \hat{E}_{\mathfrak{a}} (1/2+ it)  E_{\mathfrak{a}} (z, 1/2+ it) \left( \frac{1}{T} \int_{0}^{T} \frac{2d(f_x,t)}{x^{1/2}} dr \right) dt.
\end{equation}
To justify this, as in the discrete spectrum we notice it is well-defined as coming from the spectral expansion of the integrated error (\ref{integratederror}). Hence, to complete the proof of Theorem \ref{result1}, we need to prove that the expansion in (\ref{continuouscontributionconj}) converges to 
\begin{eqnarray*}
\frac{|\Gamma(3/4)|^2}{\pi^{3/2}} \sum_{\mathfrak{a}} \hat{E}_{\mathfrak{a}} (1/2) E_{\mathfrak{a}} (z, 1/2)
\end{eqnarray*}
 as $T \to \infty$. To deal with this expansion, we need the following lemma for the Huber transform.
 \begin{lemma}\label{hubertransformconvergence}
As $T \to \infty$ we have
\begin{eqnarray*}
\lim_{T \to \infty}  \int_{-\infty}^{\infty} \frac{1}{T} \int_{0}^{T} \frac{2 d(f_x, t)}{x^{1/2}} dr dt = \frac{4}{\sqrt{\pi}} |\Gamma(3/4)|^2.
\end{eqnarray*}
\end{lemma}
\begin{proof}
Using expression (\ref{df_hypergeometric}) we write 
\begin{eqnarray} \label{thefirstintegralbeforecalcpath}
 \int_{-\infty}^{\infty} \frac{1}{T} \int_{0}^{T} \frac{2 d(f_x, t)}{x^{1/2}} dr dt &=&  \Re \left( \int_{-\infty}^{\infty}  \frac{1}{T} \int_{0}^{T}  G(t) \Gamma(it) e^{ir t}  F \left(-\frac{1}{2}, \frac{3}{2};1+it; \frac{1}{e^{2r} + 1} \right)  dr dt \right). 
\end{eqnarray}
The convergence of the above integral can be justified as above, using that the Huber transform commutes with convolution in the $x$ variable.
Let $\varepsilon>0$ be a fixed small number and $M>0$ be a fixed large number. We consider the path integral  
\begin{eqnarray} \label{firstpathintegralprincipal}
\int_{\gamma}  G(z) \Gamma(iz)  \frac{1}{T} \int_{0}^{T} e^{ir z}  F \left(-\frac{1}{2}, \frac{3}{2};1+iz; \frac{1}{e^{2r} + 1} \right) dr dz,
\end{eqnarray}
where $\gamma$ is the contour $\gamma = \bigcup_{i=1}^{6} C_i$ with
\begin{eqnarray*}
C_1 &=& [\epsilon, M], \\
C_2 &=& \{M+iv, v \in [0,1/2]\}, \\
C_3 &=& [-M+i/2,M +i/2], \\
C_4 &=& \{-M+iv, v \in [0,1/2]\}, \\
C_5 &=& [-M, -\epsilon], \\
C_6 &=& \{ \varepsilon e^{i\theta}, \theta \in [0,\pi] \},
\end{eqnarray*}
traversed counterclockwise. To calculate (\ref{firstpathintegralprincipal}) we write $G(z)$ as
\begin{eqnarray*}
G(z) =  \frac{\sqrt{2}}{\pi} \frac{ \Gamma \left(\frac{3}{4} + \frac{iz }{2}\right) \Gamma \left(\frac{3}{4} - \frac{iz}{2}\right)}{ \Gamma(3/2+iz)} \left(e^{-\frac{i \pi}{4} - \frac{\pi z}{2}}  +  e^{ \frac{i \pi}{4} + \frac{\pi z}{2}}\right),
\end{eqnarray*}
hence we see that the integrand is holomorphic inside the contour. The simple pole at $z=0$ is coming from $\Gamma(iz)$. We note that $\hbox{Res}_{z=0} \Gamma(iz) =-i$. Applying Stirling's formula and the asymptotics of the hypergeometric function (\ref{hypergeometricspecialvalue}) we deduce
\begin{eqnarray*}
\int_{C_2+C_4}  G(z) \Gamma(iz)  \frac{1}{T} \int_{0}^{T} e^{ir z}  F \left(-\frac{1}{2}, \frac{3}{2};1+iz; \frac{1}{e^{2r} + 1} \right)  dr dz &=& O \left(M^{-2} T^{-1} \right), \\
\int_{C_3}  G(z) \Gamma(iz)  \frac{1}{T} \int_{0}^{T} e^{ir z}  F \left(-\frac{1}{2}, \frac{3}{2};1+iz; \frac{1}{e^{2r} + 1} \right)  dr dz &=& O \left(T^{-1} \right).
\end{eqnarray*}
Further, as $\varepsilon \to 0$ we see that the term
\begin{eqnarray*}
\int_{C_6}  G(z) \Gamma(iz) \frac{1}{T} \int_{0}^{T} e^{ir z}  F \left(-\frac{1}{2}, \frac{3}{2};1+iz; \frac{1}{e^{2r} + 1} \right)  dr dz 
\end{eqnarray*}
converges to 
\begin{eqnarray*}
 - i \pi G(0) \frac{1}{T} \int_{0}^{T}  F \left(-\frac{1}{2}, \frac{3}{2};1; \frac{1}{e^{2r} + 1} \right) dr \hbox{Res}_{z=0}  \Gamma(iz) = - \pi G(0) (1+O(T^{-1})) .
\end{eqnarray*}
From Cauchy's Theorem we conclude
\begin{eqnarray*}
\int_{-M}^{M}  G(t) \Gamma(it)   \frac{1}{T} \int_{0}^{T}  e^{irt} F \left(-\frac{1}{2}, \frac{3}{2};1+it; \frac{1}{e^{2r} + 1} \right) dr dt &=& \pi G(0) (1+O(T^{-1})) \\
&&+ O(M^{-2} T^{-1} + T^{-1}).
\end{eqnarray*}
As $M \to \infty$ we get
\begin{eqnarray*}
 \int_{-\infty}^{\infty} \frac{1}{T} \int_{0}^{T} \frac{2 d(f_x, t)}{x^{1/2}} dr dt &=& 2 \frac{\Gamma(3/4)^2}{\Gamma(3/2)} + O(T^{-1}),
\end{eqnarray*}
and for $T \to \infty$ the statement follows.
\end{proof}

We let $\phi_{\mathcal{H}, \frak{a}} (t)$ denote the function
\begin{equation}
\phi_{\mathcal{H}, \frak{a}} (t) = \hat{E}_{\mathfrak{a}} (1/2+ it)  E_{\mathfrak{a}} (z, 1/2+ it) -  \hat{E}_{\mathfrak{a}} (1/2)  E_{\mathfrak{a}} (z, 1/2).
\end{equation}
Thus, the contribution of the cusp $\frak{a}$ in eq. (\ref{continuouscontributionconj}) can we written in the form
\begin{eqnarray} \label{splittedcontributionofcuspa} 
\begin{split}
&& \frac{1}{4\pi} \hat{E}_{\mathfrak{a}} (1/2)  E_{\mathfrak{a}} (z, 1/2) \int_{-\infty}^{\infty}   \left( \frac{1}{T} \int_{0}^{T} \frac{2 d(f_x,t)}{x^{1/2}} dr \right) dt \\
&&+ \frac{1}{4 \pi} \int_{-\infty}^{\infty} \phi_{\mathcal{H}, \frak{a}} (t)   \left( \frac{1}{T} \int_{0}^{T} \frac{2 d(f_x,t)}{x^{1/2}} dr \right) dt.
\end{split}
\end{eqnarray}
The second term of (\ref{splittedcontributionofcuspa}) can be handled using Lemma \ref{lemmacoefficentsconjugacy}. We calculate:
\begin{eqnarray*} 
\frac{1}{4 \pi} \int_{-\infty}^{\infty} \phi_{\mathcal{H}, \frak{a}} (t)   \left( \frac{1}{T} \int_{0}^{T} \frac{2 d(f_x,t)}{x^{1/2}} dr \right) dt &=&  \frac{1}{2 \sqrt{2}\pi^2}  \int_{-\infty}^{\infty} \phi_{\mathcal{H}, \frak{a}} (t) G(t) \Gamma(it) \frac{e^{itT} - 1}{itT} dt \\
&&+ O \left( \frac{1}{T}   \int_{-\infty}^{\infty}  \phi_{\mathcal{H}, \frak{a}} (t) \frac{G(t) \Gamma(it)}{(1+|t|)(2+|t|)}  dt  \right).
\end{eqnarray*}
Since $\phi_{\mathcal{H}, \frak{a}} (0)=0$, applying Theorems \ref{localweylslawperiods2} and \ref{localweylslaw} 
we conclude the bound
\begin{eqnarray*} 
\int_{-\infty}^{\infty} \phi_{\mathcal{H}, \frak{a}} (t)   \left( \frac{1}{T} \int_{0}^{T} \frac{2 d(f_x,t)}{x^{1/2}} dr \right) dt = O(T^{-1}).
\end{eqnarray*}
Hence, as $T \to \infty$ the contribution of the continuous spectrum converges to
\begin{eqnarray*} 
 \pi^{-3/2} |\Gamma(3/4)|^2 \sum_{\frak{a}} \hat{E}_{\mathfrak{a}} (1/2)  E_{\mathfrak{a}} (z, 1/2).
\end{eqnarray*}
This completes the proof of Theorem \ref{result1}.

\section{$\Omega$-results for the average error term on geodesics}  \label{section4}

In this section we give the proof of Theorem \ref{result2}. For this reason, we mollify the average of the error term on the geodesic $\ell$. Let $\psi \geq 0$ be a smooth even function compactly supported in $[-1,1]$,  such that $\hat{\psi} \geq 0$ and $\int_{-\infty}^{\infty} \psi(x) dx = 1$. For every $\epsilon >0$ we also define the family of functions $\psi_{\epsilon}(x) = \epsilon^{-1} \psi(x/\epsilon)$. We have $0 \leq \hat{\psi}_{\epsilon}(x) \leq 1$ and $\hat{\psi}_{\epsilon}(0) = 1$. As before, we study separately the contributions of the discrete and the continuous spectrum.

\subsection{The contribution of the discrete spectrum} Let us denote by $e(\mathcal{H}, R)$ the average of the normalized error term on the geodesic, evaluated at the parameter $R = \log( X+U)$, i.e.
\begin{eqnarray*}
e(\mathcal{H}, R) =: \int_{\mathcal{H}} \frac{e(\mathcal{H}, X ;z)}{X^{1/2} } d s(z),
\end{eqnarray*}
and we consider the convolution
\begin{eqnarray*}
\left( e(\mathcal{H}, \cdot ) \ast \psi_{\epsilon} \right) (R) =: \int_{-\infty}^{+\infty} \psi_{\epsilon}(R-Y)   e(\mathcal{H}, Y )d Y. 
\end{eqnarray*}
Notice that if $|e(\mathcal{H}, Y)| \leq M$ for $Y$ $\epsilon$-close to $R$, then $|\left( e(\mathcal{H}, \cdot ) \ast \psi_{\epsilon} \right) (R)| \leq M$, by the properties of $\psi(x)$. It follows that, in order to prove an $\Omega$-result for the average $\int_{\mathcal{H}} e(\mathcal{H}, X ;z)d s$, it suffices to prove an $\Omega$-result  for the convolution $(e(\mathcal{H}, \cdot) \ast \psi_{\epsilon}) (R) $. Further, using Lemma \ref{lemmacoefficentsconjugacy}, Stirling's asymptotic (\ref{asymptot}), Theorem \ref{localweylslawperiods2} and the properties of $\psi$ we calculate the contribution of the discrete spectrum in $\left( e(\mathcal{H}, \cdot) \ast \psi_{\epsilon} \right) (R)$ is given by
\begin{eqnarray*}
&&\sum_{ t_j >0}  |\hat{u}_j|^2  \Re \left(   G(t_j) \G(it_j) \int_{-\infty}^{+\infty}   \psi_{\epsilon}(Y-R)   e^{i  t_j Y}  d Y  \right) \\ &&+ O \left( \sum_{ t_j >0} |\hat{u}_j|^2 \left| \int_{-\infty}^{+\infty}  e^{-Y/2} \psi_{\epsilon}(Y-R) V(Y) e^{ i t_j Y}   d Y\right| + e^{-\sigma R}  \right) \\
&&=  \sum_{ t_j >0}   |\hat{u}_j|^2   \Re \left( G(t_j) \G(it_j) e^{i t_j R}  \right)  \hat{\psi}_{\epsilon}(t_j) + O \left( e^{- R/2} +  e^{-\sigma R}  \right),
\end{eqnarray*}
where the last estimate follows immediately from the properties of $V$. Let $A>1$. We split the sum of the above main term for $t_j \geq A$ and $t_j <A$. Using the bound 
\begin{equation} \label{psiepsilonhatusefulbound}
\hat{\psi}_{\epsilon}(t_j) = O_k ((\epsilon |t_j|)^{-k})
\end{equation}
 for every $k \geq 1$, for $t_j \geq A$ we get
\begin{eqnarray*}
\sum_{ t_j \geq A} |\hat{u}_j|^2   \Re \left( G(t_j) \G(it_j) e^{i t_j R}  \right)  \hat{\psi}_{\epsilon}(t_j) =O_k (\epsilon^{-k} A^{-k}).
\end{eqnarray*}
For the partial sum part of the series we use the following lemma:
\begin{lemma} [Dirichlet's box principle \cite{phirud}] \label{dirichletsboxprinciple} Let $r_1, r_2, ... , r_n$ be $n$ distinct real numbers and $M>0$, $T>1$. Then, there is an $R$ satisfying $M \leq R \leq M T^n$, such that
\begin{eqnarray*}
|e^{ir_jR} -1| < \frac{1}{T}
\end{eqnarray*}
for all $j=1,...,n$.
\end{lemma}
We apply Lemma \ref{dirichletsboxprinciple} to the sequence $e^{i t_j R}$ and Lemma \ref{localweylslawperiods2}. Given $T$ large we find an $R$ such that $M \ll R \ll M T^{n} \ll M T^{A^2}$. The contribution of the discrete spectrum in the convoluted error term $\left( e(\mathcal{H}, \cdot) \ast \psi_{\epsilon} \right) (R)$ takes the form
\begin{equation} \label{expansionforcoefuptoa}
\sum_{t_j<A}  |\hat{u}_j|^2    \Re \left( G(t_j) \Gamma(it_j)  \right)  \hat{\psi}_{\epsilon}(t_j)  + O_k \left( T^{-1} \log A +\epsilon^{-k} A^{-k} + e^{-\sigma R}   \right).
\end{equation}
The balance $ A \log A =T$, $\log M \asymp \epsilon^{-1}$, $\epsilon^{-2} = A$ implies $\log \log R \asymp \log (\epsilon^{-1}) $ and for $\epsilon \leq 1$ we get:
\begin{eqnarray*}
T^{-1} \log A +\epsilon^{-k} A^{-k} + e^{-\sigma R} = O(\epsilon + e^{-\sigma R}).
\end{eqnarray*}
From part $a)$ of Lemma \ref{gammalemma2} we conclude the sum in (\ref{expansionforcoefuptoa}) is positive.  On the other hand there exists one $\tau \in (0,1)$ such that $\hat{\psi}(x) \geq 1/2$ for $|x| \leq \tau$. Since $ \hat{\psi}_{\epsilon}(t_j) =  \hat{\psi} (\epsilon t_j)$, we get 
\begin{eqnarray*}
 \sum_{t_j<A}  \Re \left(  G(t_j) \Gamma(it_j)  \right)  \hat{\psi}_{\epsilon}(t_j) |\hat{u}_j|^2 
 &\gg& \sum_{t_j<\tau / \epsilon}  \Re \left( G(t_j) \Gamma(it_j)  \right)  |\hat{u}_j|^2 \\
&\gg& \sum_{t_j<\tau / \epsilon}  t_j^{-1}  |\hat{u}_j|^2.
\end{eqnarray*}
When $\G$ is cocompact or has sufficiently small Eisenstein periods in the sense of Definition \ref{sufficientmanydefinition}, we have
\begin{eqnarray*}
 \sum_{t_j<\tau / \epsilon}  t_j^{-1}  |\hat{u}_j|^2  \gg \log(\epsilon^{-1}) 
\gg \log \log R.
\end{eqnarray*}
We conclude that the contribution of the discrete spectrum in $\ e(\mathcal{H}, R) $ is $ \Omega_{+}( \log \log R)$. This implies that if $\Gamma$ is cocompact or has sufficiently small Eisenstein periods, the contribution of the discrete spectrum in $ \int_{\mathcal{H}} e(\mathcal{H}, X;z) ds$ is $\Omega_{+} (X^{1/2} \log \log \log X)$. In particular, this completes the proof of Theorem \ref{result2} for $\Gamma$ cocompact. 

\subsection{The contribution of the continuous spectrum} \label{subsection4.2}

The contribution of the continuous spectrum in $ (e(\mathcal{H}, \cdot ) \ast \psi_{\epsilon})(R)$ is given by the quantity
\begin{eqnarray*} \label{continuouscontributionconj2}
\sum_{\mathfrak{a}} \frac{1}{4\pi} \int_{-\infty}^{\infty} | \hat{E}_{\mathfrak{a}} (1/2+ it) |^2  \Re \left(  G(t) \Gamma(it) e^{i R t}            F \left(-\frac{1}{2}, \frac{3}{2};1+it; \frac{1}{e^{2R} + 1} \right)  \right)  \hat{\psi}_{\epsilon}(t) dt.
\end{eqnarray*}
The convergence of the integral is justified as in section \ref{section3}. Let $\chi_{\mathcal{H}, \frak{a}} (t)$ denote the function
$\chi_{\mathcal{H}, \frak{a}} (t) = |\hat{E}_{\mathfrak{a}} (1/2+ it)|^2 -  |\hat{E}_{\mathfrak{a}} (1/2)|^2$.
Thus the contribution of cusp $\mathfrak{a}$ in $ (e(\mathcal{H}, \cdot ) \ast \psi_{\epsilon})(R)$ splits in
\begin{eqnarray} \label{secondintegralprincipalvalue}
\begin{split}
&& \frac{|\hat{E}_{\mathfrak{a}} (1/2)|^2}{4\pi}   \int_{-\infty}^{\infty} \Re \left( G(t) \Gamma(it) e^{i R t}  F \left(-\frac{1}{2}, \frac{3}{2};1+it; \frac{1}{e^{2R} + 1} \right)     \right)  \hat{\psi}_{\epsilon}(t)  dt \\
&&+\frac{1}{4\pi} \int_{-\infty}^{\infty} \chi_{\mathcal{H}, \frak{a}} (t)  \Re \left(  G(t) \Gamma(it) e^{i R t}  F \left(-\frac{1}{2}, \frac{3}{2};1+it; \frac{1}{e^{2R} + 1} \right)  \right)  \hat{\psi}_{\epsilon}(t) dt.
\end{split}
\end{eqnarray}
Let $\gamma$ be the contour $\gamma = \bigcup_{i=1}^{6} C_i$ defined in the proof of Lemma \ref{hubertransformconvergence}. The function $\psi_{\epsilon}(x)$ is compactly supported in the interval $[-\epsilon, \epsilon]$. Appyling the Paley-Wiener Theorem \cite[Theorem~7.4]{katznelson} we deduce that the holomorphic Fourier transform of $\psi_{\epsilon}(x)$:
\begin{eqnarray*}
\hat{\psi}_{\epsilon}(z) = \int_{-\infty}^{\infty} \psi_{\epsilon}(x) e^{-ixz} dx
\end{eqnarray*} 
is an entire function of type $\epsilon$, i.e. $|\hat{\psi}_{\epsilon}(z)| \ll e^{\epsilon |z|}$, and it is square-integrable over horizontal lines:
\begin{eqnarray*}
\int_{-\infty}^{\infty} |\hat{\psi}_{\epsilon}(v +iu)|^2 dv < \infty.
\end{eqnarray*} 
For fixed $\epsilon>0$ we have
\begin{eqnarray*}
\int_{-\infty}^{\infty} |\hat{\psi}_{\epsilon}(v +iu)|^2 dv = \epsilon^{-1}  \int_{-\infty}^{\infty} |\hat{\psi} (v +i \epsilon u)|^2 dv 
\end{eqnarray*} 
and since $\int_{-\infty}^{\infty} |\hat{\psi} (v +i \epsilon u)|^2 dv$ converges uniformly to $\int_{-\infty}^{\infty} |\hat{\psi} (v)|^2 dv$ as $\epsilon \to 0$ we get 
\begin{eqnarray} \label{psiepsilonhatverticallines}
\int_{-\infty}^{\infty} |\hat{\psi}_{\epsilon}(v +i/2)^2 dv \ll \epsilon^{-1}.
\end{eqnarray}
Consider the integral 
\begin{eqnarray} \label{secondintegralprincipalvalueafterpath}
\int_{\gamma}  G(z) \Gamma(iz) e^{iR z}  F \left(-\frac{1}{2}, \frac{3}{2};1+iz; \frac{1}{e^{2R} + 1} \right) \hat{\psi}_{\epsilon}(z)  dz.
\end{eqnarray}
The integrand is holomorphic inside the contour. Working as in the proof of Lemma \ref{hubertransformconvergence} and applying Cauchy-Schwarz inequality and bound (\ref{psiepsilonhatverticallines}) for the integral over $C_3$ we deduce
\begin{eqnarray*}
\int_{-\infty}^{\infty}  G(t) \Gamma(it) e^{iR t} F \left(-\frac{1}{2}, \frac{3}{2};1+it; \frac{1}{e^{2R} + 1} \right) \hat{\psi}_{\epsilon}(t) dt &=& \pi G(0) \hat{\psi}_{\epsilon}(0) \left(1+O\left(e^{-2R}\right)\right) \\
&&+ O( \epsilon^{-1} e^{-R/2}).
\end{eqnarray*}
To finish the proof of part (a) of Theorem \ref{result2}, we notice that if 
\begin{eqnarray*}
 \int_{-T}^{T} | \hat{E}_{\mathfrak{a}} (1/2 + it)|^2 dt \ll \frac{T}{(\log T)^{1+\delta}},
\end{eqnarray*}
then the function 
\begin{eqnarray*}
H_1(t) = \chi_{\mathcal{H}, \frak{a}} (t)  G(t) \Gamma(it) F \left(-\frac{1}{2}, \frac{3}{2};1+it; \frac{1}{e^{2R} + 1} \right) \hat{\psi}_{\epsilon}(t)
\end{eqnarray*}
is in $L^1 (\mathbb{R})$ independently of $\epsilon$ and $R$. To obtain this we notice that $\chi_{\mathcal{H}, \frak{a}} (t) \Gamma(it)$ remains bounded close to $t=0$, we use the trivial bound $\hat{\psi}_{\epsilon}(t) \leq 1$, Lemma \ref{lemmacoefficentsconjugacy} and we estimate
\begin{eqnarray}
 \int_{-\infty}^{\infty} |H_1(t)| dt &\ll& \int_{-1}^{1} |H_1(t)|  dt   + \sum_{n=0}^{\infty} 2  \int_{2^n}^{2^{n+1}} |t|^{-1} | \hat{E}_{\mathfrak{a}} (1/2 + it)|^2  dt \nonumber \\
 &\ll& \int_{-1}^{1} |H_1(t)|  dt   + \sum_{n=0}^{\infty} 2^{-n} \int_{2^n}^{2^{n+1}}| \hat{E}_{\mathfrak{a}} (1/2 + it)|^2  dt \\
 &\ll& \int_{-1}^{1} |H_1(t)|  dt   + \sum_{n=0}^{\infty} \frac{1}{(n+1)^{1+\delta}} \ll 1. \nonumber
\end{eqnarray}
Applying the Riemann--Lebesgue Lemma we conclude that
\begin{eqnarray}
\lim_{R \to \infty} \int_{-\infty}^{\infty} H_1(t) e^{i R t}   dt = 0.
\end{eqnarray}
Since $\hat{\psi}_{\epsilon}(0)=1$ and $\pi G(0) = 4 \pi^{-1/2} |\Gamma(3/4)|^2$, the contribution of the continuous spectrum in $ (e(\mathcal{H}, \cdot ) \ast \psi_{\epsilon})(R)$ takes the form
\begin{eqnarray} \label{continuouscontributionconj2final}
\frac{1}{\pi^{3/2}}|\Gamma(3/4)|^2 \sum_{\mathfrak{a}} | \hat{E}_{\mathfrak{a}} (1/2) |^2 + O( \epsilon^{-1} e^{-R/2}) + o(1).
\end{eqnarray}
As in the discrete spectrum (see the balance after expansion (\ref{expansionforcoefuptoa})) we choose the balance $\epsilon^{-1} \ll \log R \ll \log \log X$. Hence (\ref{continuouscontributionconj2final}) takes the form 
\begin{eqnarray} \label{continuouscontributionconj2thelastone}
\frac{1}{\pi^{3/2}}|\Gamma(3/4)|^2 \sum_{\mathfrak{a}} | \hat{E}_{\mathfrak{a}} (1/2) |^2 + O(X^{-1/2} \log \log X ) +o(1).
\end{eqnarray}
In particular, this completes the proof of part (a) of Theorem \ref{result2}. 

To prove part (b), we first notice that the contribution from the discrete spectrum is $c(R) + O_k \left( T^{-1} \log A +\epsilon^{-k} A^{-k} + e^{-\sigma R} \right)$, where $c(R)= \Omega_{+}(1)$ if there exists one $\hat{u}_j \neq 0$ and $c(R)$ vanishes otherwise. In this case, the contribution of the continuous spectrum takes the form
\begin{eqnarray} 
\frac{1}{\pi^{3/2}}|\Gamma(3/4)|^2 \sum_{\mathfrak{a}} | \hat{E}_{\mathfrak{a}} (1/2) |^2 + O( \epsilon^{-1} e^{-R/2}) \nonumber + \epsilon^{-1}  \int_{-\infty}^{\infty} H_2(t) e^{i R t} dt
\end{eqnarray}
where, using Theorem \ref{localweylslawperiods2} and estimate (\ref{psiepsilonhatusefulbound}), we deduce that the function
$H_2(t) := \epsilon H_1(t)$
is in $L^1 (\mathbb{R})$ independently of $\epsilon$ and $R$. Applying the Riemann--Lebesgue Lemma, the contribution of the continuous spectrum becomes
\begin{eqnarray*}
\pi^{-3/2}|\Gamma(3/4)|^2 \sum_{\mathfrak{a}} | \hat{E}_{\mathfrak{a}} (1/2) |^2 +  O( \epsilon^{-1} e^{-R/2}) + \epsilon^{-1} Q(R),
\end{eqnarray*}
with $Q(R) = o(1)$ as $R \to \infty$. We choose the balance $\epsilon^{-2}= A$. For $\epsilon=\epsilon_0$ sufficiently small and fixed and letting $R, T \to \infty$ we conclude that the convoluted normalized error $ (e(\mathcal{H}, \cdot ) \ast \psi_{\epsilon})(R)$ takes the form
\begin{eqnarray*}
c(R) + \pi^{-3/2}|\Gamma(3/4)|^2 \sum_{\mathfrak{a}} | \hat{E}_{\mathfrak{a}} (1/2) |^2 +o(1), 
\end{eqnarray*}
The second summand is $\Omega_{+}(1)$ if and only if $\hat{E}_{\mathfrak{a}} (1/2) \neq 0$ for at least one cusp $\mathfrak{a}$. Part (b) now follows.

\begin{remark}
For part (a)  of Theorem \ref{result2}, even if $\Gamma$ has not sufficiently small Eisenstein periods associated to $\mathcal{H}$ but has sufficiently many cusp forms in the sense that 
\begin{equation}
\sum_{0< t_j < T} |\hat{u}_j|^2 \gg T,
\end{equation}
we can derive the $\Omega_{+}(X^{1/2} \log \log \log X)$ bound if we have a polynomial bound for the derivatives of the Eisenstein series on the critical line (see \cite[Chapter~4]{chatzthesis} for details). 
\end{remark}

\subsection{An arithmetic case: the modular group} \label{arithmeticexampleresults}

In this subsection we concentrate to $\Gamma = {\hbox{PSL}_2( {\mathbb Z})}$. The set of primitive indefinite quadratics forms $Q(x,y) = ax^2+bxy +cy^2$ in two variables (that means $(a,b,c) = 1$ and $b^2-4ac = d>0$ is not a square) is in one-to-one correspondence with the set of primitive hyperbolic elements of $\Gamma$ (see \cite[p.~232]{sarnak}). Here we briefly describe this correspondence.

The automorphs of $Q$ is the cyclic group $\hbox{Aut}(Q) \subset {\hbox{SL}_2( {\mathbb Z})}$ which fixes $Q$, under the action
\begin{eqnarray*}
 \left( \begin{array}{cc} a & b/2 \\ b/2 & c \end{array} \right) =  \gamma^{t}  \left( \begin{array}{cc} a & b/2 \\ b/2 & c \end{array} \right) \gamma.
\end{eqnarray*}
Let $M_{Q}$ be a generator of $\hbox{Aut}(Q)$. Then the correspondence $Q \to M_{Q}$ is bijective between indefinite integral quadratic forms in two variables and primitive hyperbolic elements of the modular group. Denote by $\mathcal{H}_{Q}$ the conjugacy class of $M_{Q}$ and by $\ell_{Q}$ the $M_{Q}$-invariant geodesic. Define 
$$r(Q,n) = \# ( \{ (x,y) \in \mathbb{Z}^2 : Q(x,y) = n \} / \hbox{Aut}(Q) ),$$
and let $\zeta(Q,s)$ denote the Epstein zeta function
\begin{eqnarray}
\zeta(Q,s) = \sum_{n=1}^{\infty} \frac{r(Q,n)}{n^s},
\end{eqnarray}
which is absolutely convergent in $\Re(s)>1$. Hecke proved that the Eisenstein period $\hat{E}_{\mathfrak{a}}(s)$ along a normalized segment of $\ell_{Q}$ satisfies
\begin{eqnarray} \label{heckerelation}
\hat{E}_{\mathfrak{a}}(s)  = \frac{ d^{s/2} \Gamma^2(s/2)}{\zeta(2s) \Gamma(s) } \zeta(Q,s)
\end{eqnarray}
(see \cite[eq.~(9.5)]{tsuzuki}). The functional equation of the Eisenstein series implies the functional equation of the Epstein zeta function: 
\begin{eqnarray*}
d^{(1-s)/2} \Gamma^2 \left(\frac{1-s}{2}\right) \pi^{s-1} \zeta(Q,1-s) = d^{s/2} \Gamma^2 \left(\frac{s}{2}\right) \pi^{-s} \zeta(Q,s).
\end{eqnarray*}
The functional equation and the Phragm\'en-Lindel\"of principle imply the convexity bound on the critical line:
\begin{eqnarray} \label{convexityforzetaqoncritical}
\zeta(Q,1/2+it) \ll_{\epsilon} (1+|t|)^{1/2 + \epsilon}, \quad t \in \mathbb{R}.
\end{eqnarray}
Further, for the Epstein zeta function $\zeta(Q,1/2+it)$ the following subconvexity bound holds:
\begin{eqnarray} \label{subconvexityboundtsuzuki}
\zeta(Q,1/2+it) \ll_{\epsilon} (1+|t|)^{1/3 + \epsilon}
\end{eqnarray}

To prove this, write the Epstein zeta function $\zeta(Q, s)$ as a linear combination of zeta functions $\zeta(s, \chi)$, where $\chi$ runs through the class group characters of the number field $Q(\sqrt{d})$ \cite[Ch.~12, p.~216]{iwaniec2}.
The bound (\ref{subconvexityboundtsuzuki}) now follows from the $GL(1)$-subconvexity bound over a number field and the subconvexity bound of S\"ohne \cite{sohne} for Hecke zeta functions with Gr\"ossencharacters.


In this case we deduce that $\G$ has sufficiently small Eisenstein periods; in fact
\begin{eqnarray} \label{modulargroupsmallperiods}
 \int_{-T}^{T} | \hat{E}_{\mathfrak{a}} (1/2 + it)|^2 dt \ll T^{2/3+\epsilon}
\end{eqnarray}
 for every $\epsilon>0$. To prove this, we use the bound
$|\zeta(1+ 2it)|^{-1} \ll (\log|t|)^{2/3} (\log \log |t|)^{1/3}$
as $|t| \to \infty$ \cite[Th. 8.29]{iwanieckowalski} and Stirling's formula, which imply 
\begin{eqnarray*}
 \frac{|\Gamma^2(1/4+it/2)|}{ |\Gamma(1/2+it)|} \ll (1+|t|)^{-1/2}.
\end{eqnarray*}
Thus
\begin{eqnarray*}
\hat{E}_{\mathfrak{a}}(1/2 +it) \ll (1+|t|)^{-1/2} (\log|t|)^{2/3} (\log \log |t|)^{1/3} \zeta(Q,1/2+it) \ll_{\epsilon} (1+|t|)^{-1/6 +\epsilon}
\end{eqnarray*}
for every $\epsilon >0$, and the bound (\ref{modulargroupsmallperiods}) follows. In particular, the subconvexity bound (\ref{subconvexityboundtsuzuki}) implies
\begin{eqnarray*}
\int_{\ell_Q} e(\mathcal{H}_{Q}, X;z) ds(z) = \Omega_{+}(X^{1/2} \log \log \log X).
\end{eqnarray*}

\section{Pointwise $\Omega$-results for the error term} \label{section5}

In this section we prove Propositions \ref{result3}, \ref{result4}, and hence Theorem \ref{result5}, where we consider pointwise $\Omega$-results for the error term $e(\mathcal{H}, X;z)$. We start with the discrete average. The arguments of the proofs follow the ideas from sections \ref{section3} and \ref{section4} (see \cite[Chapter 4]{chatzthesis} for detailed proofs).

\subsection{Proof of Proposition \ref{result3}: The discrete spectrum}
For $K>0$ an integer we pick equally spaced $z_1, z_2, ... , z_{K}$ points on the invariant closed geodesic $\ell$ of $\mathcal{H}$ with $\rho(z_{i+1}, z_i) = \delta$. Hence $\delta = \mu(\ell) / K$. For $R =\log (X + U)$ we define the quantity
\begin{displaymath}
N_{K} (\mathcal{H}, R)= \frac{1}{K} \sum_{m=1}^{K} \frac{e(\mathcal{H}, X;z_{m})}{X^{1/2}}
\end{displaymath}
and we consider the convolution
\begin{eqnarray*}
\left( \psi_{\epsilon} \ast N_K(\mathcal{H}, \cdot) \right) (R) &=& \int_{-\infty}^{\infty} \psi_{\epsilon}(R-Y)  N_K (\mathcal{H}, Y) dY.
\end{eqnarray*}
Using Lemma \ref{lemmacoefficentsconjugacy}, the properties of $\psi_{\epsilon}$, Theorem \ref{localweylslaw} and Theorem \ref{localweylslawperiods2} we conclude
\begin{eqnarray*}
\left( \psi_{\epsilon} \ast N(\mathcal{H}, \cdot)_{K} \right) (R) =  \sum_{t_j >0}  \hat{u}_j  \left(\frac{1}{K} \sum_{m=1}^{K} u_j(z_{m}) \right)   \Re \left(  G(t_j)  \Gamma(it_j) e^{i t_j R} \right)   \hat{\psi}_{\epsilon}(t_j) + O(e^{-\sigma R} ).
\end{eqnarray*}
For $A>1$, using Stirling's formula, Theorem \ref{localweylslaw}, Theorem \ref{localweylslawperiods2} and estimate \ref{psiepsilonhatusefulbound} for $k \geq 1$ we estimate the tail of the series for $t_j >A$ is $ O_k (\epsilon^{-k} A^{1/2-k})$. The partial sum of the series for $t_j \leq A$ can be handled as follows: by the definition of the period integral $\hat{u}_j$, as $K \to \infty$  we get
\begin{eqnarray*}
\frac{\mu(\ell)}{K} \sum_{m=1}^{K} u_j(z_{m}) =\sum_{m=1}^{K} u_j(z_{m}) \delta \to  \overline{\hat{u}_j}
\end{eqnarray*}
uniformly, for every $j=1,...,n$ (where $n$ is such that $t_n \leq A < t_{n+1}$,  hence $n \asymp A^2$). That means for every small $\epsilon_1>0$ there exists a $K_0 = K_0(\epsilon_1) \geq 1$ such that
\begin{eqnarray} \label{aproximationtoujhatdiscrete}
 \hat{u}_j  \left(\frac{1}{K} \sum_{m=1}^{K} u_j(z_{m}) \right) =  \frac{ |\hat{u}_j|^2 }{\mu(\ell)} + O \left( \epsilon_1 \hat{u}_j \right)
\end{eqnarray}
for every $K \geq K_0$. We get  
\begin{eqnarray*}
\left( \psi_{\epsilon} \ast N(\mathcal{H}, \cdot)_{K} \right) (R) &=& \frac{1}{\mu(\ell)} \sum_{t_j \leq A}  |\hat{u}_j|^2   \Re \left(  G(t_j) \Gamma(it_j) e^{i  t_j R} \right)   \hat{\psi}_{\epsilon}(t_j) \\
&&+  O_k \left( \epsilon_1\sum_{t_j \leq A}  \hat{u}_j  \Re \left(  G(t_j)  \Gamma(it_j)  e^{i t_j R} \right) \hat{\psi}_{\epsilon}(t_j) + \epsilon^{-k} A^{1/2-k} + e^{-\sigma R} \right).
\end{eqnarray*}
Using Theorem \ref{localweylslawperiods2} the $O$-term is bounded by $O(\epsilon_1 A^{1/2})$. For the main term, apply Dirichlet's principle (Lemma \ref{dirichletsboxprinciple}) to the exponentials $e^{it_j R}$. For every $M$ and $T$ we find $ M \ll R \ll M T^{A^2}$ such that
\begin{eqnarray*}
\left( \psi_{\epsilon} \ast N(\mathcal{H}, \cdot)_{K} \right) (R) &=& \frac{1}{\mu (\ell)} \sum_{t_j \leq A}  |\hat{u}_j|^2   \Re \left(  G(t_j) \Gamma(it_j) \right)   \hat{\psi}_{\epsilon}(t_j) \\
&&+  O_k (\epsilon^{-k} A^{1/2-k} + T^{-1} \log A + \epsilon_1 A^{1/2} +    e^{-\sigma R} ).
\end{eqnarray*}
The balance $\epsilon^{-1} = A^{1 -3/(2k+2)}$, $\epsilon_1 =A^{-1/2} \epsilon$ implies the $O$-term is $O( T^{-1} \log A + \epsilon +e^{-\sigma R} )$.
By Lemma \ref{gammalemma2}, the coefficients of the above sum are all positive. For the function $\psi$ we pick $\tau \in (0,1)$ such that $\hat{\psi}(x) \geq 1/2$ for $|x| \leq \tau$. 
It follows that if $\Gamma$ is cocompact or has sufficiently small Eisenstein periods we bound the above sum from below by 
\begin{eqnarray*} 
 \frac{1}{\mu(\ell)} \sum_{t_j \leq A}  |\hat{u}_j|^2   \Re \left(  G(t_j) \Gamma(it_j) \right)   \hat{\psi}_{\epsilon}(t_j) \gg \log(\epsilon^{-1}).
\end{eqnarray*}
We deduce that for every $\epsilon >0$ we can find a sufficiently large $K = K(\epsilon)$ such that
\begin{eqnarray*}
\left( \psi_{\epsilon} \ast N_K(\mathcal{H}, \cdot) \right) (R) = k(\epsilon) + O(\epsilon + e^{-\sigma R}).
\end{eqnarray*}
with $k(\epsilon) =\Omega_{+}( \log(\epsilon^{-1}))$. If $\Gamma$ is cocompact, choosing $\epsilon=\epsilon_0$ sufficiently small and $K= K(\epsilon_0)$ sufficiently large, for $R, T \to \infty$ we conclude Proposition \ref{result3} for $\G$ cocompact.

\subsection{The continuous spectrum}

The contribution of the continuous spectrum in the convolution $\left( \psi_{\epsilon} \ast N_K(\mathcal{H}, \cdot) \right) (R)$ is given by
\begin{eqnarray} \label{lastresultcofinite}
\sum_{\mathfrak{a}} \frac{1}{4\pi} \int_{-\infty}^{\infty}&&\hat{E}_{\mathfrak{a}} (1/2+ it)  \left( \frac{1}{K} \sum_{m=1}^{K} E_{\mathfrak{a}} (z_{m}, 1/2+ it) \right)  \\
&&\times \Re \left(  G(t)  \Gamma(it) F \left(-\frac{1}{2}, \frac{3}{2};1+it; \frac{1}{e^{2R} + 1} \right)e^{i tR} \right)   \hat{\psi}_{\epsilon}(t)dt. \nonumber
\end{eqnarray}
For $A>0$, by Theorem \ref{localweylslawperiods2}, asymptotic (\ref{asymptot}) and estimate (\ref{psiepsilonhatusefulbound}) it follows that the contribution of $|t| > A$ in the above integral is $O (\epsilon^{-k} A^{1/2-k})$. For $|t| \leq A$ and for any small $\epsilon_2>0$ we approximate the Eisenstein period integral as
\begin{equation} 
 \frac{1}{K} \sum_{m=1}^{K} E_{\mathfrak{a}} (z_{m}, 1/2+ it) =  \hat{E}_{\mathfrak{a}} (1/2- it) + O(\epsilon_2)
\end{equation}
for every $K \geq K_0$ with $K_0=K_0(\epsilon_2)$ sufficiently large. The contribution of the continuous spectrum (\ref{lastresultcofinite}) takes the form
\begin{eqnarray} \label{lastresultcofiniteexpanded}
&&\sum_{\mathfrak{a}} \frac{1}{4\pi} \int_{|t| \leq A} | \hat{E}_{\mathfrak{a}} (1/2+ it) |^2 \Re \left(  G(t) \Gamma(it) e^{i R t}  F \left(-\frac{1}{2}, \frac{3}{2};1+it; \frac{1}{e^{2R} + 1} \right)  \right)  \hat{\psi}_{\epsilon}(t) dt  \\
&&+O_k \left(\epsilon_2 \sum_{\mathfrak{a}} \int_{|t| \leq A}  \hat{E}_{\mathfrak{a}} (1/2+ it)G(t)  \Gamma(it) e^{i R t} F \left(-\frac{1}{2}, \frac{3}{2};1+it; \frac{1}{e^{2R} + 1} \right) \hat{\psi}_{\epsilon}(t)dt + \epsilon^{-k} A^{1/2-k}\right). \nonumber
\end{eqnarray}
By subsection \ref{subsection4.2} and Theorem \ref{localweylslawperiods2}, the first summand of (\ref{lastresultcofiniteexpanded}) takes the form
\begin{eqnarray} 
\frac{1}{\pi^{3/2}}|\Gamma(3/4)|^2 \sum_{\mathfrak{a}} | \hat{E}_{\mathfrak{a}} (1/2) |^2 + O(  \epsilon^{-1} Q_1(R) +\epsilon^{-k} A^{-k}),
\end{eqnarray}
with $Q_1(R) \to 0$ as $R \to \infty$. For the second summand of (\ref{lastresultcofiniteexpanded}), we set $\theta_{\mathcal{H}, \frak{a}} (t) = \hat{E}_{\mathfrak{a}} (1/2+ it) - \hat{E}_{\mathfrak{a}} (1/2)$ and we use the contour integral method to deduce that the contribution of the continuous spectrum in $\left( \psi_{\epsilon} \ast N_K(\mathcal{H}, \cdot) \right) (R)$ is
\begin{eqnarray*} 
\frac{1}{\pi^{3/2}}|\Gamma(3/4)|^2 \sum_{\mathfrak{a}} | \hat{E}_{\mathfrak{a}} (1/2) |^2 + O_k \left(  \epsilon^{-1} Q_1(R) + \epsilon^{-k} A^{-k+1/2} + \epsilon_2 + \epsilon_2 \epsilon^{-1} e^{-R/2} + \epsilon_2 \log A \right).
\end{eqnarray*}
Choosing $\epsilon_2 = \epsilon^2$ and $\epsilon^{-1} = A^{1-3/(2k+2)}$ as before we conclude the $O$-term is $O( \epsilon^{-1} Q_1(R) + \epsilon)$. If $\Gamma$ has at least one $\hat{u}_j \neq 0$ with $\lambda_j >1/4$ then for fixed and sufficiently small $\epsilon$ the contribution of the discrete spectrum in $\left( \psi_{\epsilon} \ast N_K(\mathcal{H}, \cdot) \right) (R)$ is $\Omega_{+} (1)$. If $\Gamma$ has at least one nonzero Eisenstein period integral then for fixed and sufficiently small $\epsilon$ we get that the contribution of the continuous spectrum in $\left( \psi_{\epsilon} \ast N_K(\mathcal{H}, \cdot) \right) (R)$ is also $\Omega_{+}(1)$. This completes the proof of Proposition \ref{result3}.

\subsection{Proof of Proposition \ref{result4}} \label{endoftheproofs}

In this subsection we prove Proposition \ref{result4}, where we study the average of a normalized error term on the geodesic $\ell$. As we have already mentioned, this completes the proof of Theorem \ref{result5}. In particular, to simplify the estimates we will prove Proposition \ref{result4} for the average 
\begin{displaymath}
M_{\mathcal{H},z} (X) = \frac{1}{Y} \int_{1}^{Y} \frac{e(\mathcal{H}, x ;z)}{x^{1/2}} d y,
\end{displaymath}
where we define $Y$ and $y$ be given by $Y= X+ \sqrt{X^2 - 1}$ and $y= x+ \sqrt{x^2 - 1}$. We will need the following lemma for the Huber transform. 

 \begin{lemma}\label{hubertransformconvergenceinxvariable}
For $y = x + \sqrt{x^2 - 1}$ we have
\begin{eqnarray*}
\lim_{Y \to \infty}  \int_{-\infty}^{\infty} \frac{1}{Y} \int_{1}^{Y} \frac{2 d(f_x, t)}{x^{1/2}} dy dt = \frac{4}{\sqrt{\pi}} |\Gamma(3/4)|^2.
\end{eqnarray*}
\end{lemma}
The proof of Lemma follows similarly with that of Lemma \ref{hubertransformconvergence}. We can now prove Proposition \ref{result4}.

\begin{proof} (of Proposition \ref{result4}). Assume first that $\Gamma$ is cocompact. We pick $z_1, z_2, ... , z_{K}$ equally spaced points on the invariant closed geodesic $\ell$ of $\mathcal{H}$ with $\rho(z_{i+1}, z_i) = \delta$.
Using Lemma \ref{lemmacoefficentsconjugacy}, Theorem \ref{localweylslaw} and Theorem \ref{localweylslawperiods2} we conclude
\begin{eqnarray} \label{seriesaverageforprop1.12}
\frac{1}{K} \sum_{m=1}^{K} M_{\mathcal{H},z_{m}} (X) =  \sum_{t_j >0}  \hat{u}_j  \left(\frac{1}{K} \sum_{m=1}^{K} u_j(z_{m}) \right)   \Re \left(  G(t_j)  \Gamma(it_j) \frac{1}{Y} \int_{1}^{Y} e^{i t_j r} dy \right) + O(Y^{-\sigma} ),
\end{eqnarray}
For $A>1$, we use Theorem \ref{localweylslawperiods2} and we apply the estimate (\ref{psiepsilonhatusefulbound}) to bound the tail of the series in (\ref{seriesaverageforprop1.12}) for $t_j \geq A$ by $O (A^{-1/2})$. For the partial sum of the series, we approximate the period integral $\hat{u}_j$ uniformly, for every $j=1,...,n$ (where $n \asymp A^2$). For any $\epsilon_1>0$ we find a $K_0 = K_0(\epsilon_1) \geq 1$ such that for every $K \geq K_0$:
\begin{eqnarray} \label{aproximationtoujhatdiscrete2}
 \hat{u}_j  \left(\frac{1}{K} \sum_{m=1}^{K} u_j(z_{m}) \right) =  \frac{ |\hat{u}_j|^2 }{\mu(\ell)} + O \left( \epsilon_1 \hat{u}_j \right).
\end{eqnarray}
We get  
\begin{eqnarray*}
\frac{1}{K} \sum_{m=1}^{K} M_{\mathcal{H},z_{m}} (X) &=& \frac{1}{\mu(\ell)} \sum_{t_j < A}  |\hat{u}_j|^2  \Re \left(  G(t_j)  \Gamma(it_j) \frac{Y^{it_j}}{1+it_j} \right) +  O ( Y^{-1} + \epsilon_1 + A^{-1/2} + Y^{-\sigma} ).
\end{eqnarray*}
For the main term, apply Dirichlet's principle (Lemma \ref{dirichletsboxprinciple}) to the exponentials $e^{it_j R} = Y^{it_j}$. For each $T$ we can find $ R \ll T^{A^2}$ such that
\begin{eqnarray*}
\frac{1}{K} \sum_{m=1}^{K} M_{\mathcal{H},z_{m}} (X) &=& \frac{1}{\mu(\ell)} \sum_{t_j < A}  |\hat{u}_j|^2  \Re \left(  \frac{G(t_j)  \Gamma(it_j) }{1+it_j} \right) +  O ( T^{-1} +  \epsilon_1 + A^{-1/2} + Y^{-\sigma} ).
\end{eqnarray*}
By Theorem \ref{localweylslawperiods2}, as $A \to \infty$ the sum remains bounded and, for $\Gamma$ cocompact, there exist infinitely many $j$'s such that $\hat{u}_j \neq 0$. By Lemma \ref{gammalemma2}, all the nonzero terms are negative. Hence, there exists an $A_0$ such that for every $A \geq A_0$:
\begin{equation} \label{finalsumlowerbound}
 \left| \sum_{t_j < A}  |\hat{u}_j|^2  \Re \left(  \frac{G(t_j)  \Gamma(it_j) }{1+it_j} \right) \right|  \gg 1.
\end{equation}
For $T, Y$ and $A$ fixed and sufficiently large and $\epsilon_1$ fixed and sufficiently small, we find a $K= K_0$ fixed such that 
\begin{eqnarray*}
\frac{1}{K} \sum_{m=1}^{K} M_{\mathcal{H},z_{m}} (X) &=& \Omega_{-}(1). 
\end{eqnarray*}
Notice that the lower bound (\ref{finalsumlowerbound}) holds if and only if there exists at least one nonzero $\hat{u}_j$ with $\lambda_j>1/4$. 

Assume now that $\Gamma$ is not cocompact. In this case, the contribution of the discrete spectrum in 
$$\frac{1}{K} \sum_{m=1}^{K} M_{\mathcal{H},z_{m}} (X)$$
 is given by
\begin{eqnarray} \label{continuouscontributionlastproposition}
\frac{1}{K} \sum_{m=1}^{K} \sum_{\mathfrak{a}} \frac{1}{4\pi} \int_{-\infty}^{\infty} \frac{1}{Y} \int_{0}^{Y} \frac{2 d(f_x, t)}{x^{1/2}} dy \hat{E}_{\mathfrak{a}}(1/2+it) E_{\mathfrak{a}}(z_m, 1/2+it) dt.
\end{eqnarray}
We cut the integral for $|t| \leq A$ and $|t|>A$. In the interval $|t| \leq A$ we approximate the Eisenstein period $\epsilon_2$-close. Applying Lemma \ref{hubertransformconvergenceinxvariable} and following a standard calculation, expansion (\ref{continuouscontributionlastproposition}) takes the form
\begin{eqnarray*}
\frac{|\Gamma(3/4)|^2}{ \pi^{3/2}} \sum_{\mathfrak{a}} |\hat{E}_{\mathfrak{a}} (1/2) |^2 &+& \Re \left(\sum_{\mathfrak{a}} \frac{1}{4\pi} \int_{-\infty}^{\infty} \chi_{\mathcal{H}, \mathfrak{a}} (t) \frac{G(t) \Gamma(it)}{1+it} Y^{it} dt \right) \\
&+& O(A^{-1/2} + \epsilon_2 + Y^{-1}),
\end{eqnarray*}
with $K = K(\epsilon_2, A)$. Since for $\Gamma$ all the Eisenstein periods $\hat{E}_{\mathfrak{a}} (1/2)$ vanish, applying Riemann--Lebesgue Lemma for the second term the proposition follows for $A, Y$ sufficiently large and $\epsilon_2$ sufficiently small.
\end{proof}

\section{Upper bounds on geodesics}

In this section, we apply the key observation arising in the spectral theory of the conjugacy problem, that is the slower divergence for the sums of period integrals of Theorem \ref{localweylslawperiods2}, to the error terms of both the classical problem (described in subsection \ref{subsectiononeone}) and the conjugacy class problem. In particular, for the error $e(X;z,w)$ we prove the following average result.

\begin{theorem} \label{theoremclassicalongeodesics} Let $\ell_0$ be a closed geodesic of $\GmodH$ and $e(X;z,w)$ be the error term of the classical counting problem.  Then
\begin{eqnarray*}
\int_{\ell_0} e(X;z,w) d s(w) = O_{\ell_0} ( X^{1/2} \log X).
\end{eqnarray*}
\end{theorem}
The proof of this result is similar to the proof for the classical pointwise bound $O(X^{2/3})$. The standard idea here is again to approximate the kernel defined $k(u) = \chi_{[0,(x-2)/4]}$ by appropriate step functions $k_{\pm}(u)$ and use the observation
\begin{eqnarray*}
\sum_{|t_j| < T} \frac{u_j (z) \hat{u}_j}{t_j^{3/2}} \ll \log T.
\end{eqnarray*}
Similarly, for the error term $e(\mathcal{H}, X;z)$ of the conjugacy class problem we can deduce the upper bound
\begin{eqnarray*}
\int_{\ell_0} e(\mathcal{H}, X;z) d s(z) = O_{\ell_0}( X^{1/2} \log X).
\end{eqnarray*}
Since the proof of this bound is similar with that of Theorem \ref{theoremclassicalongeodesics}, it is omitted. 
\begin{proof} (of Theorem \ref{theoremclassicalongeodesics}) The proof follows the steps of the proof for the classical pointwise bound $O(X^{2/3})$, sketched in \cite[Ch.~12, p.~173]{iwaniec}.
 Assume first the cocompact case. Define the functions $k_{-}(u) \leq k(u) \leq k_{+}(u)$ by
\begin{equation}\label{kplus}
k_{+}(u) = \left\{ \begin{array}{lcl}
1, & \mbox{for} & u \leq \frac{X-2}{4},
\\ \displaystyle\frac{-4u}{Y} + \frac{X+Y-2}{Y}, & \mbox{for} & \frac{X-2}{4} \leq u \leq \frac{X+Y-2}{4},
\\ 0, & \mbox{for} &  \frac{X+Y-2}{4} \leq u,
\end{array} \right. 
\end{equation}
\begin{equation}\label{kminus}
k_{-}(u) = \left\{ \begin{array}{lcl}
1, & \mbox{for} & u \leq \frac{X-Y -2}{4}, 
\\ \displaystyle \frac{-4u}{Y} + \frac{X-2}{Y}, & \mbox{for} & \frac{X-Y-2}{4} \leq u \leq \frac{X-2}{4},
\\ 0, & \mbox{for} & \frac{X-2}{4} \leq u.
\end{array} \right. 
\end{equation}
We denote their Selberg/Harish-Chandra transform by $h_{\pm}(t)$. Using equations \cite[p.~2,~eq.(1.2)]{chatz} we get
\begin{eqnarray*}
 e(X;z,w) \ll \sum_{t_j \in \mathbb{R}-\{0\}} h_{\pm}(t_j) u_j(z) \overline{u_j(w)} + O( Y +X^{1/2} ).
\end{eqnarray*}
Hence, using estimates \cite[p.~173, eq.~(12.9)]{iwaniec} we conclude
\begin{eqnarray}\label{lastexpressionclassical}
 \int_{\ell_0} e(X;z,w) d s (w)  &\ll& \sum_{t_j \in \mathbb{R}-\{0\}} h_{\pm}(t_j) u_j(z) \hat{u}_j + O(Y +X^{1/2}) \nonumber \\
 &\ll& X^{1/2} \sum_{t_j} |t_j|^{-5/2} \min \{|t_j|, X Y^{-1} \} |u_j(z)| |\hat{u}_j| + O( Y +X^{1/2} ).
\end{eqnarray}
Applying Cauchy-Schwarz inequality, local Weyl's laws for the Maass forms $u_j(z)$ and Theorem \ref{localweylslawperiods2} for the periods $\hat{u}_j$, we deduce that (\ref{lastexpressionclassical}) is bounded by
\begin{eqnarray*}
X^{1/2} \sum_{t_j \leq X/Y} |t_j|^{-3/2}  |u_j(z)| |\hat{u}_j| + X^{1/2} \sum_{t_j >X/Y} |t_j|^{-5/2}  \frac{X}{Y} |u_j(z)| |\hat{u}_j|
\ll X^{1/2} \log (X/Y) &+& X^{1/2}.
\end{eqnarray*}
We conclude 
\begin{eqnarray*}
 \int_{\ell_0} e(X;z,w) d s(w)  \ll X^{1/2} \log (X/Y) + Y +X^{1/2}
\end{eqnarray*}
and the statement follows for $Y=X^{1/2}$. For the cofinite case, the result follows similarly, using the relevant bounds for the Eisenstein series and their period integrals.
\end{proof}

\end{document}